\documentclass{tran-l}

\usepackage{amsmath,amsthm,amssymb,mathrsfs, enumerate, verbatim}
\usepackage{lscape,chngcntr,color}
\usepackage[noadjust]{cite}
\usepackage[breaklinks=true]{hyperref}

\theoremstyle{plain}

\newtheorem{theorem}{\textrm{\textbf{Theorem}}}[section]
\newtheorem{utheorem}{\textrm{\textbf{Theorem}}}

\newtheorem{corollary}[theorem]{\textrm{\textbf{Corollary}}}
\newtheorem{proposition}[theorem]{\textrm{\textbf{Proposition}}}
\newtheorem{lemma}[theorem]{\textrm{\textbf{Lemma}}}

\theoremstyle{definition}

\newtheorem{definition}[theorem]{\textrm{\textbf{Definition}}}

\newtheorem{remark}[theorem]{\normalfont{\textit{Remark}}}

\theoremstyle{remark}

\numberwithin{equation}{section}

\def\sgn {\mathop{\rm sgn}}
\def\rk {\mathop{\rm rank}}
\def\R{\mathbb{R}}
\def\F{\mathbb{F}}
\def\Q{\mathbb{Q}}
\def\Z{\mathbb{Z}}
\def\N{\mathbb{N}}
\def\br{\mathbb{S}}
\def\bp{\mathbb{P}}
\def\ree{\mathop{\rm Re}}
\DeclareMathOperator{\Id}{\ensuremath{Id}}

\newcommand{\block}[1]{\ensuremath{{\rm diag} \left(#1 \right)}}

\begin{document}
\title[Preserving positivity for rank-constrained matrices]{Preserving
positivity\\ for rank-constrained matrices}

\thanks{The authors were partially supported by the following: US Air
Force Office of Scientific Research grant award FA9550-13-1-0043, US
National Science Foundation under grant DMS-0906392, DMS-CMG 1025465,
AGS-1003823, DMS-1106642, DMS-CAREER-1352656, Defense Advanced Research
Projects Agency DARPA YFA N66001-11-1-4131, the UPS Foundation,
SMC-DBNKY, and an NSERC postdoctoral fellowship}

\author{Dominique Guillot}
\address{Department of Mathematical Sciences, University of Delaware,
Newark, Delaware 19716}
\email{dguillot@udel.edu}

\author{Apoorva Khare}
\address{Departments of Mathematics and Statistics, Stanford University,
Stanford, California 94305}
\email{khare@iisc.ac.in}

\author{Bala Rajaratnam}
\address{Department of Statistics, University of California, Davis,
Caiifornia 95616}
\email{brajaratnam01@gmail.com}

\date{June 11, 2015 and, in revised form, September 11, 2015}



\vspace*{-1mm}\begin{abstract}
Entrywise functions preserving the cone of positive semidefinite matrices
have been studied by many authors, most notably by Schoenberg
[Duke Math.~J.~9, 1942] and Rudin [Duke Math.~J. 26, 1959].
Following their work, it is well-known that entrywise functions
preserving Loewner positivity in all dimensions are precisely the
absolutely monotonic functions. However, there are strong theoretical and
practical motivations to study functions preserving positivity in a fixed
dimension $n$. Such characterizations for a fixed value of $n$ are
difficult to obtain, and in fact are only known in the $2 \times 2$ case.
In this paper, using a novel and intuitive approach, we study entrywise
functions preserving positivity on distinguished submanifolds inside the
cone obtained by imposing rank constraints. These rank constraints are
prevalent in applications, and provide a natural way to relax the elusive
original problem of preserving positivity in a fixed dimension. In our
main result, we characterize entrywise functions mapping $n \times n$
positive semidefinite matrices of rank at most $l$ into positive
semidefinite matrices of rank at most $k$ for $1 \leq l \leq n$ and $1
\leq k < n$. We also demonstrate how an important necessary condition for
preserving positivity by Horn and Loewner [Trans.~Amer.~Math.~Soc.~136,
1969] can be significantly generalized by adding rank constraints.
Finally, our techniques allow us to obtain an elementary proof of the
classical characterization of functions preserving positivity in all
dimensions obtained by Schoenberg and Rudin.
\end{abstract}
\maketitle

\vspace*{-3mm}
\setcounter{tocdepth}{1}
\tableofcontents
\counterwithin{table}{section}

\section{Introduction and main results}\label{sec:intro}

\textit{This version contains a few edits to the published paper, which
are minor in nature and mostly clarify the statements of results. The
proofs remain virtually unchanged. A list of these edits is on the final
page.}\medskip

The study of entrywise functions mapping the space of positive
semidefinite matrices into itself has been the focus of a concerted
effort throughout the past century (see e.g. Schoenberg
\cite{Schoenberg42}, Rudin \cite{Rudin59}, Herz \cite{Herz63}, Horn
\cite{Horn}, Christensen and Ressel \cite{Christensen_et_al78}, Vasudeva
\cite{vasudeva79}, FitzGerald and Horn \cite{FitzHorn}, FitzGerald,
Micchelli, and Pinkus \cite{fitzgerald}, Hiai \cite{Hiai2009}, Guillot
and Rajaratnam \cite{Guillot_Rajaratnam2012b}, Guillot, Khare and
Rajaratnam \cite{GKR-crit-2sided,Guillot_Khare_Rajaratnam2012} and
others). Following the work of Schoenberg and Rudin, it is well-known
that functions $f:(-1,1) \rightarrow \mathbb{R}$ such that $f[A] :=
(f(a_{ij}))$ is positive semidefinite for all positive semidefinite
matrices $A = (a_{ij})$ of all dimensions with entries in $(-1,1)$ are
necessarily analytic with nonnegative Taylor coefficients; i.e., they are
absolutely monotonic on the positive real axis. The converse follows
easily from the Schur product theorem.

On the other hand, one is often interested in studying functions that
preserve positivity for a fixed dimension $n$. In this case, it is
unnecessarily restrictive to characterize functions that preserve
positivity in all dimensions. Results for fixed $n$ are available only
for $n=2$, in which case entrywise functions mapping $2 \times 2$
positive semidefinite matrices into themselves have been characterized by
Vasudeva (see \cite[Theorem 2]{vasudeva79}) in terms of multiplicatively
mid-convex functions.
Finding tractable descriptions of the functions that preserve positive
semidefinite matrices in higher dimensions is far more involved; see
\cite[Theorem 1.2]{Horn} for partial results for arbitrary but fixed $n
\geq 3$. To the authors' knowledge, no full characterization is known for
$n \geq 3$.


The primary goal in this paper is to investigate entrywise functions
mapping $n \times n$ positive semidefinite matrices of a fixed dimension
$n$ and rank at most $l$ into positive semidefinite matrices of rank at
most $k$ for given integers $1 \leq k,l \leq n$. Introducing such rank
constraints is very natural from a modern applications perspective.
Indeed, in high-dimensional probability and statistics, it is common to
apply entrywise functions to regularize covariance/correlation matrices
in order to improve their properties (e.g., condition number, Markov
random field structure, etc.); see \cite{bickel_levina, GKR-crit-2sided,
Guillot_Khare_Rajaratnam2012, Guillot_Rajaratnam2012,
Guillot_Rajaratnam2012b, hero_rajaratnam, Hero_Rajaratnam2012,
Li_Horvath, Zhang_Horvath}. In such settings, the rank of a
covariance/correlation matrix corresponds to the sample size used to
estimate it. Many modern-day applications require working with
covariance/correlation matrices arising from small samples, and so these
high-dimensional matrices are very often rank-deficient in practice.
Applying functions entrywise is a popular way to increase the rank of
these matrices. For many downstream applications, it is a requirement
that the regularized covariance/correlation matrices be positive
definite. It is thus very important and useful for applications to
understand how the rank of a matrix is affected when a given function is
applied to its entries, and whether Loewner positivity is preserved.

Our approach yields novel and explicit characterizations for functions
preserving positivity for a fixed dimension, under various rank
constraints. In particular, we show that a ``special family'' of matrices
of rank at most $2$ plays a fundamental role in preserving positivity. 
Furthermore, our techniques yield solutions to other characterization
problems, in particular, those involving rank constraints without the
positivity requirement. For instance, we provide characterizations of
entrywise functions mapping symmetric $n \times n$ matrices of rank at
most $l$ into matrices of rank at most $k$ (for any $1 \leq k,l < n$).

Finally, our methods can also be used to solve the original problem of
characterizing entrywise functions that preserve positivity \textit{for
all dimensions}, thereby providing a more direct and intuitive proof of
the results by Schoenberg, Rudin, Vasudeva, and others.\medskip

\noindent \textit{Notation.}
To state our main results, some notation and definitions are needed.
These are now collected together here, for the convenience of the
reader.

\begin{definition}\label{D11}
Let $\mathbb{R} \supset \mathbb{Z} \supset \mathbb{N}$ denote
the real numbers, the integers, and the positive integers, respectively.
Let $I \subset \mathbb{R}$. Define $\br_n(I)$ to be the set of $n \times
n$ symmetric matrices with entries in $I$, and $\bp_n(I) \subset
\br_n(I)$ to be the subset of symmetric positive semidefinite matrices.
Let $\rk A$ denote the rank of a matrix $A$. Now define: 
\begin{align}
\br_n^k(I) &:= \{A \in \br_n(I) : \rk A \leq k\}, \\
\bp_n^k(I) &:= \{A \in \bp_n(I) : \rk A \leq k\}. 
\end{align}
Next, if $f: I \subset \R \rightarrow \R$ and $A = (a_{ij})$ is a matrix
with entries in $I$, denote by $f[A]$ the matrix obtained by applying $f$
to every entry of $A$, i.e., $f[A] := (f(a_{ij}))$. Finally, denote by
$f[-] : \br_n(I) \to \br_n$ the map sending $A \in \br_n(I)$ to $f[A]$.
\end{definition}

When $I = \mathbb{R}$, we denote $\br_n^k(I)$ and
$\bp_n^k(I)$ by $\br_n^k$ and $\bp_n^k$ respectively. Note that
\begin{equation}
\br_n(I) = \br_n^n(I), \qquad \bp_n(I) = \bp_n^n(I), 
\end{equation}
and when $I \subset [0, \infty)$, $\br_n^1(I) = \bp_n^1(I)$.
Also, in what follows, given (possibly scalar) matrices $A_1, \dots, A_n$
of arbitrary orders, denote by $A_1 \oplus \dots \oplus A_n$ the
corresponding block diagonal matrix $\block{A_1, \dots, A_n}$. Note that
this differs from the Kronecker sum of matrices.

Next, given a set of vectors or matrices $A_1, \dots, A_m$ of equal
orders, denote their entrywise product by $A_1 \circ \cdots \circ A_m$,
which is a matrix of the same order with $(i,j)$th entry equal to
$\prod_{k=1}^m (A_k)_{ij}$. Given a vector or matrix $A = (a_{ij})$ and
$\alpha \in \R$ such that $a_{ij}^\alpha$ is defined for all $i,j$,
define the $\alpha$th Hadamard power of $A$ to be $A^{\circ \alpha} :=
(a_{ij}^\alpha)_{i,j}$. Given $\alpha \in \R$, define the even and odd
extensions of the power functions $f_\alpha(x) := x^\alpha$, to the
entire real line, as follows: 
\begin{align}
\begin{aligned}
& \phi_\alpha(0) = \psi_\alpha(0) := 0, \qquad \phi_\alpha(x) :=
|x|^\alpha,\\
& \psi_\alpha(x) := \sgn(x) |x|^\alpha, \qquad
\forall \alpha \in \R, \ x \in \R \setminus \{ 0 \}.
\end{aligned}
\end{align}

Finally, given a function $f : I \to \R$, and a vector or matrix $A =
(a_{ij})$ with entries $a_{ij} \in I \subset \R$, define $f[A] :=
(f(a_{ij}))_{i,j}$.

\begin{remark}
Note that if the entrywise function $f[-]$ maps $n \times n$ positive
semidefinite matrices of rank exactly $l$ into matrices of rank $k$ (for
$1 \leq l,k \leq n$), and $f$ is continuous, then $f[-]$ necessarily maps
$\bp_n^l$ into $\bp_n^k$.
Moreover, we observe that there are no (continuous) entrywise maps $f[-]$
sending $\bp_n^l$ into matrices of rank bounded below by $k \geq 2$,
since $f[c{\bf 1}_{n \times n}]$ has rank at most $1$. Here ${\bf 1}_{n
\times n}$ denotes the matrix with all entries equal to $1$. For these
reasons we will study functions sending $\bp_n^l$ to $\bp_n^k$.
\end{remark}

We now state the three main results of the paper. The first result
characterizes entrywise functions mapping rank $1$ matrices to rank at
most $k$ matrices, under the mild hypothesis that the function $f$ admits
at least $k$ nonzero derivatives of some orders at the origin.

\begin{utheorem}[Rank $1$, fixed dimension]\label{thm:sum_powers}
Let $0 < R \leq \infty$, $I = [0, R)$ or $(-R,R)$, and $f: I \rightarrow
\mathbb{R}$. Fix $1 \leq k < n$, and suppose $f$ admits at least $k$
nonzero derivatives at $0$.
\begin{enumerate}
\item Then $f[-] : \bp_n^1(I) \to \br_n^k$ if and only if $f$ is a
polynomial with exactly $k$ nonzero coefficients. 
\item Similarly, $f[-] : \bp_n^1(I) \to \bp_n^k$ if and only if $f$ is a
polynomial with exactly $k$ nonzero coefficients which are all positive. 

\item Suppose $f$ admits at least $n-1$ nonzero derivatives at zero, of
orders $0 \leq m_1 < \cdots < m_{n-1}$. If $f[-] : \bp_n^1(I) \to \bp_n$,
then writing $f$ in its Taylor expansion at $0$:
\begin{equation}
f(x) = \sum_{i=1}^{n-1} \frac{f^{(m_i)}(0)}{m_i!} x^{m_i} + x^{m_{n-1}}
h(x),
\end{equation}

\noindent the Taylor coefficients $\frac{f^{(m_i)}(0)}{m_i!}$ are
positive, and $h([0,R)) \subset [0,\infty)$.
\end{enumerate}
\end{utheorem}

Note that throughout this paper, we take the \textit{derivatives of $f$}
to include the zeroth derivative function $f(x)$. For instance, according
to our convention, the function $f(x) = 1+x$ has two nonzero derivatives
at the origin.

The intuitive approach adopted in this paper to prove Theorem
\ref{thm:sum_powers} yields rich rewards in tackling more challenging
problems. One of the primary goals of this paper is to classify all
functions which take $n \times n$ matrices with rank at most $l$ to
matrices of rank at most $k$. Our main theorem in this paper completely
classifies the functions $f$ such that $f[-] : \bp_n^l(I) \to \br_n^k$ or
$\bp_n^k$ for $l \geq 2$, under the stronger assumption that $f$ is $C^k$
on $I$. Surprisingly, when $k \leq n-3$ and $f[-] : \bp_n^l([0,R)) \to
\bp_n^k$, the $C^k$ assumption is not required. 

\begin{utheorem}[Higher rank, fixed dimension]\label{Tltok}
Let $0 < R \leq \infty$ and $I = [0, R)$ or $(-R,R)$. Fix integers $n
\geq 2$, $0 \leq k < n-1$, and $2 \leq l \leq n$. Suppose $f \in C^k(I)$.
Then the following are equivalent: 
\begin{enumerate}
\item $f[A] \in \br_n^k$ for all $A \in \bp_n^l(I)$; 
\item $f(x) = \sum_{t=1}^r a_t x^{i_t}$ for some $a_t \in \R$ and some
distinct $i_t \in \Z_{\geq 0}$ such that
\begin{equation}\label{eqn:sum_binom}
\sum_{t=1}^r \binom{i_t+l-1}{l-1} \leq k. 
\end{equation}
\end{enumerate}
Similarly, $f[-]: \bp_n^l(I) \to \bp_n^k$ if and only if $f$ satisfies
(2) and $a_i \geq 0$ for all $i$. Moreover, if $I = [0,R)$ and
$k=0$ or $k \leq n-3$, then the assumption that $f \in C^k(I)$ is not
required.
\end{utheorem}

Note that when $l=1$, to obtain Theorem \ref{thm:sum_powers}, it suffices
to assume $f$ has $k$ nonzero derivatives of arbitrary orders at only the
origin. For higher rank $l>1$, we prove Theorem \ref{Tltok} under the
stronger assumption of $f$ being $C^k$ on all of $I$.

Our last main result concerns entrywise functions preserving positivity
for all dimensions, as in the classical case studied by Schoenberg,
Rudin, Vasudeva, and many others. We demonstrate that functions
preserving positivity over a small family of rank $2$ matrices of all
dimensions are automatically analytic with nonnegative Taylor
coefficients. By contrast, classical results are generally proved under
the far stronger assumption that the function preserves positivity for
all positive semidefinite matrices. Before we state the result, first
recall the notion of an absolutely monotonic function.

\begin{definition}
Let $I \subset \mathbb{R}$ be an interval with interior $I^\circ$. A
function $f \in C(I)$ is said to be \emph{absolutely monotonic} on $I$ if
$f \in C^\infty(I^\circ)$ and $f^{(k)}(x) \geq 0\ \forall x \in I^\circ$
and every $k \geq 0$. 
\end{definition}

\begin{utheorem}[Rank $2$, arbitrary dimension]\label{thm:rank2_AM}
Let $0 < R \leq \infty$, $I = (0,R)$, and $f: I \rightarrow \mathbb{R}$.
Then the following are equivalent: 
\begin{enumerate}
\item For all $n \geq 1$, $f[a {\bf 1}_{n \times n} + u u^T] \in \bp_n$
for every $a \in [0,R)$ and $u \in \R^n$ such that $a + u_iu_j \in I$.
\item For all $n \geq 1$, $f[A] \in \bp_n$ for every $A \in \bp_n(I)$;
\item The function $f$ is absolutely monotonic on $I$. 
\end{enumerate}
\end{utheorem}

Theorem \ref{thm:rank2_AM} is a significant refinement of the original
problem, in which one studies entrywise functions which preserve Loewner
positivity among positive semidefinite matrices of all orders, and with
no rank constraints. Moreover, condition (1) is a new and much simpler
characterization of preserving positivity for all $\bp_n$, in the sense
that it simplifies condition (2) significantly. The equivalence $(2)
\Leftrightarrow (3)$ has been known for some time in the literature in
related settings. Vasudeva showed in \cite{vasudeva79} that $(2)
\Leftrightarrow (3)$ for $I = (0, \infty)$, whereas Schoenberg and Rudin
showed the same result for $I = (-1,1)$. See Theorems \ref{Therz} and
\ref{vasudeva_rank2}.

\begin{remark}
Unlike Theorems \ref{thm:sum_powers} and \ref{Tltok},
Theorem \ref{thm:rank2_AM} makes no continuity or differentiability
assumptions on $f$. Also note that the family of rank 2 matrices in part
(1) of Theorem \ref{thm:rank2_AM} is a one-dimensional extension of the
set of rank 1 matrices. This is in some sense the smallest family of
matrices for which the result can hold. More precisely, the corresponding
result to Theorem \ref{thm:rank2_AM} for the smaller set of rank $1$
matrices is false. For example, $f(x) := x^\alpha$ for $\alpha > 0$ sends
$\bp_n^1([0,\infty))$ to $\bp_n^1$ for all $n$; however, $f$ is not
absolutely monotonic unless $\alpha \in \N$. In this sense, Theorem
\ref{thm:rank2_AM} provides minimal assumptions under which entrywise
functions preserving positivity are absolutely monotonic.
\end{remark}

The rest of the paper is organized as follows: we begin by reviewing
previous work on functions preserving positivity in Section \ref{Salter},
and show how several of these results can be extended to more general
settings. In Section \ref{Srank1} we develop a general three-step
approach for studying entrywise functions mapping $\bp_n^1$ to $\br_n^k$,
and use it to prove Theorem \ref{thm:sum_powers}. The results of Section
\ref{Srank1} are then extended in Sections \ref{Srank2} and
\ref{Srank2_2} to study the general case of entrywise functions mapping
$\bp_n^l$ into $\br_n^k$. Finally we demonstrate in Section \ref{Sam} how
our results and techniques can be used to obtain novel and intuitive
proofs of the classical results of Schoenberg and Rudin. Along the way,
we obtain several characterizations of entrywise functions preserving
Loewner positivity in a variety of settings.

\section{Previous results and extensions}\label{Salter}

We begin by reviewing known characterizations of functions preserving
positivity, and extending them to other settings (other domains,
Hermitian matrices). We first recall a fundamental result by Schoenberg
and Rudin, characterizing entrywise functions preserving \textit{all}
positive semidefinite matrices. This celebrated characterization has been
studied and generalized by many authors under various restrictions; only
the most general version is presented here. 

\begin{theorem}[see Schoenberg \cite{Schoenberg42}, Rudin \cite{Rudin59},
Herz \cite{Herz63}, Christensen and Ressel
\cite{Christensen_et_al78}]\label{Therz}
Given $0 < R \leq \infty$, $I = (-R,R)$, and $f : I \to \R$, the
following are equivalent:
\begin{enumerate}
\item $f[A] \in \bp_n$ for all $A \in \bp_n(I)$ and all $n \geq 1$. 

\item $f$ is analytic on the disc $D(0,R) := \{ z \in \mathbb{C} : |z| <
R \}$ and absolutely monotonic on $(0,R)$; i.e., $f$ has a convergent
Taylor series on $D(0,R)$ with nonnegative coefficients.
\end{enumerate}
\end{theorem}

Similarly, the following result was shown by Vasudeva \cite{vasudeva79}
for $I = (0,\infty)$ and also follows from Horn \cite[Theorem 1.2]{Horn}.

\begin{theorem}[Vasudeva, {\cite[Theorem
6]{vasudeva79}}]\label{vasudeva_rank2}
Given $I = (0,\infty)$ and $f : I \to \R$, the following are equivalent:
\begin{enumerate}
\item $f[A] \in \bp_n$ for all $A \in \bp_n(I)$ and all $n \geq 1$.

\item $f$ can be extended analytically to $\mathbb{C}$ and is absolutely
monotonic on $I$.
\end{enumerate}
\end{theorem}

Note that the assertions in Theorems \ref{Therz} and \ref{vasudeva_rank2}
are very similar, but for different domains of definition. We now extend
Theorem \ref{vasudeva_rank2} to general nonnegative intervals. 

\begin{theorem}\label{thm:unif}
Let $0 \leq a < b \leq \infty$. Assume $I = (a,b)$ or $I =
[a,b)$ and let $f: I \rightarrow \mathbb{R}$. Then each of the
following assertions implies the next one: 
\begin{enumerate}
\item $f[A] \in \bp_n$ for all $A \in \bp_n(I)$ and all $n \geq 1$;

\item The function $f$ can be extended analytically to $D(0,b)$ and 
\begin{equation}
f(z) = \sum_{n=0}^\infty c_n z^n, \qquad z \in D(0,b) 
\end{equation}
for some $c_n \geq 0$; 

\item $f$ is absolutely monotonic on $I$.
\end{enumerate}  
If furthermore, $I = [0,b)$, then $(3) \Rightarrow (2)$ and so all the
assertions are equivalent. 
\end{theorem}

\begin{proof}
Without loss of generality, we can assume $b < \infty$ (otherwise, the
result follows by considering bounded intervals contained in $I$). That
$(2) \Rightarrow (1)$ holds follows by the Schur product theorem
and the continuity of the eigenvalues.

We now prove (1) $\Rightarrow$ (3). Consider the function $g:
(-b,b) \rightarrow \mathbb{R}$ given by
\begin{equation}\label{eqn:defn_g}
g(x) := f\left(\frac{b-a}{2b} x + \frac{a+b}{2}\right). 
\end{equation} 
Assume first that $I = (a,b)$.
Consider the linear change of variables $T : (-b,b) \to (a,b)$, given by
\[
T(x) := \left( \frac{b-a}{2b} \right) x + \frac{a+b}{2}.
\]

\noindent Thus $T[B] \in \bp_n((a,b))$ for all $B \in \bp_n((-b,b))$.
This implies that $g[-] = (f \circ T)[-]$ maps $\bp_n((-b,b))$ into
$\bp_n$ for all $n \geq 1$.
Thus, by Theorem \ref{Therz}, $g$ is analytic on $D(0,b)$ and absolutely
monotonic on $(0,b)$. It follows that $f$ is analytic on $(a,b)$ and
absolutely monotonic on $((a+b)/2,b)$. Repeating the above construction
for all $I = (a,b_0)$ with $a < b_0 < b$, it follows that $f$ is
absolutely monotonic on $I = (a, b)$ as desired. 

We next assume that $I = [0, b)$ and (1) holds. Then, clearly, (1) also
holds for matrices with entries in $(0,b)$. Thus, from above, $f$ is
analytic and absolutely monotonic on $(0, b)$. To prove that $f$ is
continuous at $0$, consider the matrix
\begin{equation}\label{eqn:matrixA_continuity}
A = \begin{pmatrix}
2 & 1 & 1 \\ 1 & 2 & 0 \\ 1 & 0 & 2
\end{pmatrix} \in \bp_3, 
\end{equation}

\noindent and, for $t > 0$, let $A_t = tA$. Since $A_t \oplus {\bf
0}_{(n-3) \times (n-3)} \in \bp_n((0,b))$ for $t$ small enough, applying
$f$ entrywise, we conclude by (1) that $f[A_t] \in \bp_3$. Also, since
$f$ is absolutely monotonic on $(0,b)$, it is nonnegative and increasing
there, and so $f^+(0) := \lim_{x \rightarrow 0^+} f(x)$ exists and is
nonnegative. Moreover, 
\begin{equation}\label{eqn:f_plus}
\lim_{t \rightarrow 0^+} f[A_t] = \begin{pmatrix}
f^+(0) & f^+(0) & f^+(0)\\
f^+(0) & f^+(0) & f(0)  \\
f^+(0) & f(0) & f^+(0)
\end{pmatrix} \in \bp_3. 
\end{equation}

\noindent This is possible only if the principal minors of this matrix
are nonnegative. It follows that $0 \leq f(0) \leq f^+(0)$ and the
determinant, which equals $-f^+(0)(f^+(0) - f(0))^2$, is nonnegative. But
then $f^+(0) = f(0)$, i.e., $f$ is right-continuous at $0$. This proves
(1) $\Rightarrow$ (3) if $I = [0, b)$. The result for $I = [a,b)$ follows
from that for $[0, b-a)$ by using the translation $g(x) = f(x+a)$, as
above.

Finally, by standard results from classical analysis (see Theorem
\ref{thm:abs_monotonic_equiv}), the implication (3) $\Rightarrow$ (2)
holds when $0 \in I$. 
\end{proof}

When the function $f$ is analytic, Theorem \ref{Therz} can easily be
extended to complex-valued functions as follows:

\begin{theorem}
Suppose $0 < R \leq \infty$ and $f : D(0,R) \to \mathbb{C}$ is analytic.
The following are equivalent:
\begin{enumerate}
\item $f[A]$ is Hermitian positive semidefinite for all Hermitian
positive semidefinite matrices $A$ with entries in $D(0,R)$.

\item $f$ is absolutely monotonic on $D(0,R)$; i.e., $f(z) = \sum_{n \geq
0} a_n z^n$ for real scalars $a_n \geq 0$.
\end{enumerate}
\end{theorem}

The proof is an easy exercise using Theorem \ref{Therz} and the
uniqueness principle for analytic functions. Note that there also exist
nonanalytic functions preserving positivity in arbitrary dimensions; see
\cite[Theorem 3.1]{fitzgerald} for the complete classification of these
maps.


As mentioned earlier, very few results characterizing functions
preserving positivity on $\bp_n$ for a fixed dimension $n$ exist. To our
knowledge, the only such known result is for $n=2$ by Vasudeva
\cite[Theorem 2]{vasudeva79}, and it characterizes functions mapping
$\bp_2(0,\infty)$ to $\bp_2$. We now prove an extension of this result to
more general intervals.

\begin{theorem}\label{thm:vasudeva:M2}
Let $0 < b \leq \infty$, and $I = (a,b)$ for $|a| \leq b$, or $I = [a,b)$
for $-b \leq a \leq 0$. Given $f: I \rightarrow \mathbb{R}$, the
following are equivalent:
\begin{enumerate}
\item $f[A] \in \bp_2$ for every $2 \times 2$ matrix $A
\in \bp_2(I)$.
\item $f$ satisfies 
\begin{equation}\label{Emidconvex}
f(\sqrt{xy})^2 \leq f(x) f(y)  \qquad \forall x,y \in I \cap [0, \infty)
\end{equation}
and 
\begin{equation}\label{eqn:increasing}
|f(x)| \leq f(y) \qquad \forall\ |x| \leq y \in I.
\end{equation}
\end{enumerate}

\noindent In particular, if $(1)$ holds, then either $f \equiv 0$ on $I
\setminus \{ -b \}$ or $f(x) > 0$ for all $x \in I \cap (0, \infty)$.
Moreover $f$ is continuous on $I \cap (0,\infty)$.
\end{theorem}

Note that the condition $|a| \leq b$ is assumed in Theorem
\ref{thm:vasudeva:M2} because no $2 \times 2$ matrix in $\bp_2(I)$ can
have any entry in $(-\infty, -b)$. Also see \cite[Lemma 2.1]{Hiai2009}
for the special case where $I = (0,R)$ for some $0 < R \leq \infty$.

\begin{proof}
Let $A = \begin{pmatrix} p & q\\q & r \end{pmatrix}$.
Clearly, (1) holds if and only if $f(q)^2 \leq f(p)f(r)$ whenever $q^2
\leq pr$ for $p,q,r \in I$, and $f(p) \geq 0$ for $p \in I \cap
[0,\infty)$. Thus, if (1) holds, then so does equation
\eqref{Emidconvex}. Equation \eqref{eqn:increasing} follows easily by
considering the matrix $\displaystyle \begin{pmatrix} y & x \\ x & y
\end{pmatrix}$ for $|x| \leq y \in I$. 

Conversely, assume (2) holds. Setting $x=0$ in \eqref{eqn:increasing}
shows that $f(y) \geq |f(0)| \geq 0$ whenever $y \in I \cap (0,\infty)$.
Now if $q^2 \leq pr$ with $p,r \geq 0$ and $p,q,r \in I$, then applying
\eqref{eqn:increasing} with $x=q$ and $y = |q|$, we obtain $f(q)^2 \leq
f(|q|)^2$. Therefore, by \eqref{eqn:increasing} and \eqref{Emidconvex}, 
\[
f(q)^2 \leq f(|q|)^2 \leq f(\sqrt{pr})^2 \leq f(p) f(r).
\]

\noindent This proves $(2) \Rightarrow (1)$. 

Next, suppose (1) holds and $f(x) = 0$ for some $x \in I \cap
(0,\infty)$. We claim that $f \equiv 0$ on $I \cap [0,b)$, which proves
via \eqref{eqn:increasing} that $f \equiv 0$ on $I \setminus \{ -b \}$.
To see the claim, first define $x_0 := \sup \{ x \in I \cap (0,\infty) \
: \ f(x) = 0 \}$.
Then $f$ vanishes on $I \cap [0,x_0)$ by \eqref{eqn:increasing}.
We now produce a contradiction if $x_0 < b$, which proves the claim, and
hence all of (2). Indeed if $x_0 < y \in I$, then choose any $x_1 \in I
\cap (x_0^2 / y, x_0)$. Thus, $\sqrt{x_1 y} \in (x_0,y) \subset I$, so by
\eqref{Emidconvex},
\[
f(\sqrt{x_1 y})^2 \leq f(x_1) f(y) = 0.
\]

\noindent This contradicts the definition of $x_0$, and proves the claim.

Finally, define $a' := \inf (I \cap (0,\infty))$, and $g(x) := \ln
f(e^x)$ for $x \in (\ln a',\ln b)$. It is clear that $g$ is nondecreasing
and mid(point)-convex on the interval $(\ln a', \ln b) \subset \R$,
whenever $f$ satisfies \eqref{Emidconvex}.
(By a midpoint convex function $g : J \to \R$ we mean $g((x+y)/2) \leq
(g(x) + g(y))/2$ for $x,y \in J$.)
Hence by \cite[Theorem 71.C]{roberts-varberg}, $g$ is necessarily
continuous (and hence convex) on $(\ln a', \ln b)$. We conclude that $f$
is continuous on $(a',b) = I \cap (0, \infty)$.
\end{proof}

We note that functions preserving other forms of positivity have been
studied by various authors in many settings, including by Ando and Hiai
\cite{ando-hiai}, Ando and Zhan \cite{AndoZhan99}, Bharali and Holtz
\cite{bharali}, Bhatia and Karandikar \cite{bhatia-rlk}, de Pillis
\cite{depillis_69}, Hansen \cite{Hansen92}, Marcus and Katz
\cite{Marcus_Katz_69}, Marcus and Watkins \cite{Marcus_watkins_71},
Michhelli and Willoughby \cite{micchelli_et_al_1979}, Thompson
\cite{Thompson_61}, Zhang \cite{Zhang_2012}, and in previous work
\cite{Guillot_Khare_Rajaratnam-CEC}--\cite{Guillot_Rajaratnam2012b} by
the authors.

\section{Preserving positivity under rank constraints I:\\
The rank $1$ case}\label{Srank1}

We begin by studying entrywise functions mapping $\bp_n^1$ into
$\br_n^k$. In Subsection \ref{S3step} we introduce a three-step approach
for studying entrywise functions $f[-]: \bp_n^1 \to \br_n^k$ for $1 \leq
k < n$ and use it to prove Theorem \ref{thm:sum_powers}. Next, in
Subsection \ref{subsec:rank1} we study entrywise functions $f[-]: \bp_n^1
\to \bp_n$ and demonstrate important connections with the Laplace
transform.
Finally, in Subsection \ref{S2sided}, we study generalizations of Theorem
\ref{thm:sum_powers} involving the two-sided extensions of the power
functions.

\subsection{A three-step approach: functions preserving positivity of
rank 1 matrices}\label{S3step}

This subsection is devoted to proving Theorem \ref{thm:sum_powers}. To do
so, we adopt the following general and intuitive three-step strategy:  
\begin{enumerate}
\item[(S1)] We begin by characterizing entrywise functions mapping
$\bp_n^1$ into $\br_n^1$.
\item[(S2)] Assuming $f(0) = 0$, we show that the rank of $f[A]$ remains
the same if $f$ is replaced by $f(x) / x^r$ for a suitable $r > 0$.
\item[(S3)] Assuming $f(0) \ne 0$, we prove that if $f$ maps $\bp_n^1$
into $\br_n^k$, then $f-f(0)$ maps $\bp_{n-1}^1$ into $\br_{n-1}^{k-1}$.
\end{enumerate}

\noindent Several of our characterization results (including Theorem
\ref{thm:sum_powers}) follow by repeatedly applying steps (S2) and (S3)
until the result reduces to characterizing functions mapping $\bp_n^1$
into $\br_n^1$; and then we apply (S1) to obtain the desired result. We
thus prove three propositions in this section, which pertain to the above
three steps. 

We start by recalling the following basic result from linear algebra,
which we shall invoke frequently in this paper (see e.g. \cite[Theorem
IV.16]{dickson}). Note that the result is seemingly presented in
loc.~cit.~only over any subfield of $\mathbb{C}$, but in fact holds over
any field.

\begin{lemma}[{\cite[Theorem IV.16]{dickson}}]\label{lem:rank_minors}
Let $A$ be a symmetric $n \times n$ matrix over any field, and let $1
\leq r \leq n$. Then the following are equivalent: 
\begin{enumerate}
\item $\rk A \leq r$; 
\item All $(r+1) \times (r+1)$ minors of $A$ (when defined) vanish;
\item All $(r+1) \times (r+1)$ and $(r+2) \times (r+2)$ principal minors
of $A$ (when defined) vanish.
\end{enumerate}
\end{lemma}

We now characterize functions mapping $\bp_n^1$ into $\br_n^1$, as in the
first step (S1) outlined at the beginning of this subsection.

\begin{proposition}\label{prop:charac_rank1}
Let $0 < b \leq \infty$, and $I = (a,b)$ for $|a| \leq b$, or $I = [a,b)$
for $|a| < b$.
Let $f : I \to \R$. Then:
\begin{enumerate}
\item $f[A] \in \br_2^1$ for every $A \in \bp_2^1(I)$ if and only if
$f(\pm\sqrt{xy})^2 = f(x) f(y)$ for all $x, y \in I \cap [0,\infty)$ such
that $\pm \sqrt{xy} \in I$.

\item Suppose $0 \in I$, $f(p) = 0$ for some $p \in I \setminus \{ 0 \}$,
and $f[-] : \bp_n^1(I) \to \br_n^1$ for some $n \geq 2$. Then $f \equiv
0$ on $I$.
\end{enumerate}
\end{proposition}

\begin{proof}\hfill
\begin{enumerate}
\item This follows immediately since $\det f[A] = 0$.

\item Suppose $f(p) = 0$ for some $p \in I \setminus \{ 0 \}$.
We first claim that there exists a sequence of positive numbers $p_m \in
I \cap (0,\infty)$ increasing to $b$ such that $f(p_m) = 0$. Define the
numbers $p_m$ and the matrices $A_m$ inductively as follows:
\begin{align*}
p_1 := &\ |p| \in (0,b), \qquad p_{m+1} := \sqrt[4]{p_m^3 b} \quad
\forall m \in \N,\\
A_m := &\ \begin{pmatrix} p_m & p_{m+1}\\ p_{m+1} & p_{m+1}^2 / p_m
\end{pmatrix} \oplus {\bf 0}_{(n-2) \times (n-2)}.
\end{align*}

\noindent It is easily verified that $p_m$ lies in $I$ and increases to
$b$, and $A_m \in \bp_n^1(I)$ for all $m$. Now applying $f$ entrywise to
the matrix $\begin{pmatrix} |p| & p \\ p & |p| \end{pmatrix} \oplus {\bf
0}_{(n-2) \times (n-2)} \in \bp_n^1(I)$ shows by (1) that $f(p_1) = 0$.
Next, using that $f[A_m] \in \br_n^1$ implies inductively that
$f(p_{m+1}) = 0$ for all $m$, which shows the claim.

Now let $q \in I$. Then $q \in I \cap (-p_m,p_m)$ for some $m$, in which
case applying $f$ entrywise to the rank $1$ matrix
$\begin{pmatrix} p_m & q \\ q & q^2/p_m \end{pmatrix} \oplus {\bf
0}_{(n-2) \times (n-2)} \in \bp_n^1(I)$ shows by (1) that $f(q) =
0$.\qedhere
\end{enumerate}
\end{proof}

\begin{remark}
Note that the functions that send $\bp_2^1(I)$ to $\bp_2^1$ instead of
the larger set $\br_2^1$ can be characterized by the same conditions as
in Proposition \ref{prop:charac_rank1}(1), together with the fact that $f
\geq 0$ on $I \cap [0,\infty)$. The proof is similar to the one above.
\end{remark}

We now present an elegant characterization of continuous functions
mapping $\bp_n^1(I)$ into $\br_n^1$. To state the result, we first define
the even and odd extensions of the power functions $f_\alpha(x) :=
x^\alpha$, to the entire real line, as follows: 
\begin{equation}\label{eqn:phi_psi_alpha}
\phi_\alpha(x) := |x|^\alpha, \qquad
\psi_\alpha(x) := \sgn(x) |x|^\alpha, \qquad \forall \alpha > 0, \ x \in
\R.
\end{equation}

\begin{lemma}\label{lem:continuous_rank1}
Let $0 < b \leq \infty$, and $I = (a,b)$ for $|a| \leq b$, or $I = [a,b)$
for $|a| < b$.
Let $n \geq 2$ and $f : I \to \R$ be continuous. Then the following are
equivalent: 
\begin{enumerate}
\item $f[A] \in \br_n^1$ for all $A \in \br_n^1(I)$; 
\item $f[A] \in \br_n^1$ for all $A \in \bp_n^1(I)$; 
\item There exists $c \in \R$ such that either $f \equiv c$ on $I$ or
$f(x) \equiv c \phi_\alpha(x)$ or $c \psi_\alpha(x)$ for some $\alpha >
0$.
\end{enumerate}
Moreover, $f[-] : \bp_n^1(I) \to \bp_n^1$ if and only if (3) holds with
$c \geq 0$.
\end{lemma}

\begin{proof}
Clearly $(3) \Rightarrow (1) \Rightarrow (2)$ by (a variant
of) Proposition \ref{prop:charac_rank1}(1). We now show that $(2)
\Rightarrow (3)$. If (2) holds, then the function $f$ satisfies 
\begin{equation*}
f(\sqrt{xy}) = \pm\sqrt{|f(x)|} \sqrt{|f(y)|} \qquad \forall x,y \in I
\cap [0,\infty)
\end{equation*}

\noindent since $\bp_2^1(I)$ embeds into $\bp_n^1(I)$ via padding by
zeros. Since $f$ is continuous on $I \setminus \{ 0 \}$, it follows that
for every $0 \leq \lambda \leq 1$,
\begin{equation}\label{Econt}
f(x^\lambda y^{1-\lambda}) = \pm |f(x)|^\lambda |f(y)|^{1-\lambda} \qquad
\forall x,y \in I \cap (0,\infty). 
\end{equation}

\noindent Equivalently, the function $g(x) := \ln |f(e^x)|$ satisfies
\begin{equation}
g(\lambda x + (1-\lambda)y) = \lambda g(x) + (1-\lambda) g(y) \qquad
\forall x,y \in \ln (I \cap (0,\infty)).
\end{equation}

\noindent Thus, $g(x) = \alpha x + \beta$ for some constants
$\alpha,\beta \in \R$. As a consequence, $f(x) = c x^\alpha$ for all $x
\in I \cap (0,\infty)$, where $|c| = e^\beta$.

It remains to compute $f$ on $I \cap (-\infty,0)$. Suppose $x \in I \cap
(-\infty,0)$; then applying $f$ entrywise to the matrices
\begin{equation}\label{Erank1ex}
\begin{pmatrix} |x| & x\\ x & |x| \end{pmatrix} \oplus {\bf 0}_{(n-2)
\times (n-2)} \in \bp_n^1(I)
\end{equation}

\noindent shows that $f(x) = \pm f(|x|)$.
There are now two cases: first if
$f(x) = cx^\alpha$ on $I \cap (0,\infty)$, with $c=0$ or $\alpha = 0$,
then since $f$ is continuous on $I$, it is easy to check that $f$ is
constant on $I$. The second case is if $c \neq 0$ and
$\alpha \neq 0$. Then $\alpha >0$ as $f$ is continuous on $I$ and $0 \in
I$ by assumption. Moreover, on the interval $I \cap (-\infty,0)$, the
function $f(x) / |x|^\alpha$ is continuous and has image in $\{ -c, c
\}$, by the above analysis. Hence $f(x) / |x|^\alpha$ is constant on $I
\cap (-\infty,0)$.
Thus $f(x) \equiv c |x|^\alpha$ or $-c |x|^\alpha$ for all $0 > x \in I$,
which shows (3). Given these equivalences, it is clear that $f[-]:
\bp_n^1(I) \to \bp_n^1$ if and only if $c \geq 0$. 
\end{proof}

Lemma \ref{lem:continuous_rank1} addresses the first step (S1) outlined
at the beginning of this subsection. The following proposition will play
a central role later and addresses the second step (S2).

\begin{proposition}\label{prop:rank_onesided}
Let $0 < R \leq \infty$, $I = [0, R)$ or $(-R,R)$, and $f, g: I
\rightarrow \mathbb{R}$ such that $g(x)$ is nonzero whenever $x$ is
nonzero. Assume $c:= \lim_{x \rightarrow 0, x \in I} f(x) / g(x)$ exists
and define 
\begin{equation}\label{Ehat0}
h_c(x) := \begin{cases}
\displaystyle \frac{f(x)}{g(x)}, & x \not= 0\\
c, & x=0.
\end{cases}
\end{equation}
Fix integers $n \geq 2$ and $1 \leq k \leq n$.
\begin{enumerate}
\item Suppose $g[-]: \bp_n^1(I) \to \br_n^1$. If $f[-]: \bp_n^1(I) \to
\br_n^k$, then $h_c[-] : \bp_n^1(I) \to \br_n^k$.  The converse holds if
$f(0) = cg(0)$.
\item Suppose $g[-]: \bp_n^1(I) \to \bp_n^1$, and $f[-]: \bp_n^1(I) \to
\bp_n$. Then $c \geq 0$ and $h_c[A] \in \bp_n$ for all $A \in
\bp_n^1(I)$.
\end{enumerate} 
\end{proposition}

\begin{proof}
We prove the result for $I = (-R,R)$; the proof is similar for $I =
[0,R)$. Note that (1) is trivial if $k=n$; thus, we assume that $k<n$ and
prove (1). First note that if $A,B \in \br_n$ and $B$ is a rank $1$
matrix with nonzero entries, then $B^{\circ (-1)} :=
(b_{ij}^{-1})_{i,j=1}^n$ also has rank $1$. Hence, since $\rk A \circ B
\leq (\rk A)(\rk B)$ (see \cite[Theorem 5.1.7]{Horn_and_Johnson_Topics})
we obtain
\begin{align}\label{eqn:rank_equal}
\begin{aligned}
\rk A = &\ \rk ((A \circ B) \circ B^{\circ (-1)}) \leq \rk (A \circ B)
\rk B^{\circ (-1)}\\
= &\ \rk A \circ B \leq \rk A \cdot \rk B = \rk A.
\end{aligned}
\end{align}

\noindent We conclude that $\rk A \circ B = \rk A$.

Now suppose $f,g$ satisfy the assumptions. Then by Proposition
\ref{prop:charac_rank1}(2) applied to $g$, $g(x) \not= 0$ for all $0 \neq
x \in I$ since $g \not\equiv 0$. Thus given $c \in \R$, and nonzero $u_i$
such that $u_i u_j \in I\ \forall i,j$, we have
\begin{equation}\label{eqn:h_c_schur}
h_c[u u^T] = \left(\frac{f(u_i u_j)}{g(u_i u_j)}\right)_{i,j=1}^n = f[u
u^T] \circ (1/g)[u u^T].
\end{equation}

\noindent It follows from \eqref{eqn:rank_equal} that
\begin{equation}\label{Erank}
\rk h_c[u u^T] = \rk f[u u^T] \leq k.
\end{equation}

Next, suppose $u \in \R^n$ such that $u_i u_j \in I\ \forall i,j$, and
let $0 < \epsilon < \sqrt{R}$. Define $u^\epsilon \in \R^n$ to be the
vector with coordinates $u_i + \epsilon \delta_{u_i, 0}$, where
$\delta_{a,b}$ denotes the Kronecker delta.
Hence by \eqref{Erank}, $\rk h_{c}[u^\epsilon u^{\epsilon T}] = \rk
f[u^\epsilon u^{\epsilon T}] \leq k$. Now, by Lemma
\ref{lem:rank_minors}, every $(k+1) \times (k+1)$ minor of
$h_{c}[u^\epsilon u^{\epsilon T}]$ is equal to zero for all $\epsilon >
0$. By the continuity of $h_{c}$ at $0$ and continuity of the determinant
function, it follows that the same is true for $h_{c}[u u^T]$. Thus, $\rk
h_{c}[u u^T] \leq k$ for all $u \in \R^n$ such that $u_i u_j \in I\
\forall i,j$, proving the first result.
Conversely, if $f(0) = c g(0)$ and $h_c[-]: \bp_n^1(I) \to \br_n^k$, then
using a similar argument as above, we obtain $f[-]: \bp_n^1(I) \to
\bp_n$. 

We now prove the second part. Note first that $f(x), g(x) \geq 0$ for all
$x \in I \cap [0, \infty)$ since $f[A], g[A] \in \bp_n$ for all $A \in
\bp_n^1(I)$. Thus, for every $\epsilon \in I \cap [0, \infty)$,
$(f/g)(\epsilon) \geq 0$ and so $c = \lim_{\epsilon \to 0^+}
(f/g)(\epsilon) \geq 0$. Now let $A = u u^T \in \bp_n^1(I)$. If $u_i \neq
0$ for all $i$, then $h_c[A] \in \bp_n$ by \eqref{eqn:h_c_schur} and the
Schur product theorem. The general case follows by a limiting argument,
replacing $u$ by $u^\epsilon$ as above. 
\end{proof}

Our last proposition in this section addresses the third step
(S3) outlined at the beginning of the present subsection.

\begin{proposition}\label{prop:general_rank_reduction}
Let $0 < R \leq \infty$, $I = [0, R)$ or $(-R,R)$, and $h : I \rightarrow
\mathbb{R}$ be such that $h(0) \ne 0$. Fix integers $n \geq 2$ and $1
\leq k,l \leq n$. Consider the following statements: 
\begin{enumerate}
\item $h[A] \in \br_n^k$ for all $A \in \bp_n^l(I)$;  
\item $(h-h(0))[A] \in \br_{n}^{k-1}$ for all $A \in
\bp_{n}^l(I)$. 
\item $(h-h(0))[A] \in \br_{n-1}^{k-1}$ for all $A \in
\bp_{n-1}^l(I)$. 
\end{enumerate}
Then $(2) \Rightarrow (1) \Rightarrow (3)$. If $k < n-1$ then $(1)
\Rightarrow (2)$. The implications $(1) \Rightarrow (2)$ (when $k < n-1$)
and $(1) \Rightarrow (3)$ also hold upon replacing the sets $\br_n^k,
\br_n^{k-1}, \br_{n-1}^{k-1}$ by $\bp_n^k, \bp_n^{k-1}, \bp_{n-1}^{k-1}$
respectively.
\end{proposition}

\begin{proof}
That $(2) \Rightarrow (1)$ is clear.
We now show that $(1) \Rightarrow (3)$.
Note that the statement is trivial if $k=n$ so we assume $1 \leq k < n$.
To do so, we fix $A \in \bp_{n-1}^l(I)$ and without loss of generality,
consider any $k \times k$ minor $M$ of $A$. Then $\begin{pmatrix} M &
{\bf 0}_{k \times 1}\\ {\bf 0}_{1 \times k} & 0\end{pmatrix}$ is a $(k+1)
\times (k+1)$-minor of the matrix $B := \begin{pmatrix} A & {\bf
0}_{(n-1) \times 1}\\ {\bf 0}_{1 \times (n-1)} & 0\end{pmatrix} \in
\bp_n^l(I)$. Consequently by applying (1) to $B$,
\[
\det \begin{pmatrix} h[M] & h(0) {\bf 1}_{k \times 1}\\ h(0) {\bf 0}_{1
\times k} & h(0)\end{pmatrix} = 0.
\]

\noindent Equivalently, subtracting the last column from every other
column, we obtain 
\begin{equation*}
\det \begin{pmatrix}
(h-h(0))[M] & h(0) {\bf 1}_{k \times 1} \\ {\bf 0}_{1 \times k} & h(0)
\end{pmatrix} = h(0) \cdot \det (h-h(0))[M] = 0. 
\end{equation*}

\noindent We conclude that $\det (h-h(0))[M] = 0$ for all $k \times k$
minors $M$ of $A$. In particular, $\rk (h-h(0))[A] < k$ by Lemma
\ref{lem:rank_minors}(2). This concludes the proof of $(1) \Rightarrow
(3)$.

We now show that $(1) \Rightarrow (2)$ when $k < n-1$. If (1) holds,
then by Lemma \ref{lem:rank_minors}(3), (2) holds if and only if every
principal $(k+i) \times (k+i)$ minor of $(h-h(0))[A]$ vanishes for each
$A \in \bp_n^l(I)$ and $i=0,1$. To show that this is indeed the case, fix
$A \in \bp_n^l(I)$
and a subset $J \subset \{ 1, \dots, n \}$ of $k+i$
indices. Without loss of generality, consider the principal submatrix
$A_J$ formed by the rows and columns of $A$ corresponding to $J$.
Let $A' := A_J \oplus {\bf 0}_{(n-k-i) \times (n-k-i)} \in \bp_n^l(I)$.
By (1) and Lemma \ref{lem:rank_minors}(3), the leading principal $(k+i+1)
\times (k+i+1)$ minor of $h[A']$ vanishes. In other words, 
\begin{equation*}
\det h[A_J \oplus {\bf 0}_{1 \times 1}] = \det \begin{pmatrix}
h[A_J] & h(0) {\bf 1}_{(k+i) \times 1} \\
h(0) {\bf 1}_{1 \times (k+i)} & h(0)
\end{pmatrix} = 0. 
\end{equation*}

\noindent As in the previous case, it follows that every principal $(k+i)
\times (k+i)$ minor of $(h-h(0))[A]$ vanishes for every $A \in
\bp_n^l(I)$ and $i=0,1$. Therefore, by Lemma \ref{lem:rank_minors}(3),
$(h-h(0))[A] \in \br_n^{k-1}$ for every $A \in \bp_n^l(I)$, which proves
(2). 

Finally, we show that $(1) \Rightarrow (2)$ (when $k<n-1$) and $(1)
\Rightarrow (3)$ when the $\br$-sets are replaced by the $\bp$-sets. We
first claim that for all $n_1, n_2 \in \N$, $c \in \R$, and $B \in
\bp_{n_1}(\R)$,
\begin{equation}\label{Eboyd}
B_c := \begin{pmatrix} B & c {\bf 1}_{n_1 \times n_2}\\ c {\bf 1}_{n_2
\times n_1} & c {\bf 1}_{n_2 \times n_2}\end{pmatrix} \in \bp_{n_1 +
n_2}(\R) \quad \Longleftrightarrow \quad c \geq 0,\ B - c {\bf 1}_{n_1
\times n_1} \in \bp_{n_1}(\R).
\end{equation}

\noindent The claim \eqref{Eboyd} is obvious for $c=0$; thus, we now
assume that $c \neq 0$. Let $C := c {\bf 1}_{n_2 \times n_2}$ and $D := c
{\bf 1}_{n_2 \times n_1}$.
Then $C^\dagger = \frac{1}{c \cdot n_2^2} \mathbf{1}_{n_2 \times n_2}$,
and we have the decomposition
\[
B_c = \begin{pmatrix} B & D^T\\D & C \end{pmatrix} =
\begin{pmatrix} \Id_{n_1} & D^T C^\dagger\\0 & \Id_{n_2} \end{pmatrix}
\begin{pmatrix} B - D C^\dagger D^T & 0\\0 & C \end{pmatrix}
\begin{pmatrix} \Id_{n_1} & 0\\ C^\dagger D & \Id_{n_2} \end{pmatrix},
\]

\noindent where $\Id_n$ denotes the $n \times n$ identity matrix. By
Sylvester's law of inertia, $B_c$ is positive semidefinite if and only if 
$B - D C^\dagger D^T = B - c \mathbf{1}_{n_1 \times n_1}$ and $C$ are
positive semidefinite. This proves the claim. 

Now suppose for the remainder of the proof that $h[-]:\bp_n^l(I) \to
\bp_n^k$. Observe that $h(0) > 0$ because $h[{\bf 0}_{n \times n}] \in
\bp_n^k$.
We now assume $1 \leq k < n-1$ and show that the modified statement of
$(3)$ holds, i.e., $(h-h(0))[A] \in \bp_{n-1}^{k-1}$ for all $A \in
\bp_{n-1}^l(I)$.
Indeed, given $A \in \bp_{n-1}^l(I)$, it follows from $(1) \Rightarrow
(3)$ that $(h - h(0))[A] \in \br_{n-1}^{k-1}$. Applying the claim
\eqref{Eboyd} with $B = h[A]$ and $c = h(0)$, it follows that $(h -
h(0))[A] \in \bp_{n-1}$ as well, proving that $(1) \Rightarrow (3)$ for
the $\bp$-sets. 

Finally, suppose $h[-]:\bp_n^l(I) \to \bp_n^k$ and $A \in \bp_{n}^l(I)$.
We now show that\break
$(h-h(0))[A] \in \bp_{n}^{k-1}$. Indeed,
it follows from the $(1) \Rightarrow (2)$ implication that $(h - h(0))[A]
\in \br_{n}^{k-1}$.
Since $(h-h(0))[A]$ is singular, it suffices to show that all its $n_1
\times n_1$ principal minors are nonnegative for $1 \leq n_1 \leq n-1$.
Let $C$ be any $n_1 \times n_1$ principal submatrix of $A$. Applying the
claim \eqref{Eboyd} with $B = h[C]$ and $c = h(0)$, it follows that
$(h-h(0))[C] \in \bp_{n_1}$. This concludes the proof.
\end{proof}

In the special case where $l=1$ and $k=2$, Proposition
\ref{prop:general_rank_reduction} immediately characterizes the functions
$f$ mapping $\bp_n^1$ to $\br_n^2$ under the assumption $f(0) \not= 0$:

\begin{corollary}
Let $n \geq 3$, $I = [0, R)$ or $(-R,R)$ for some $0 < R \leq \infty$,
$f: I \to \R$ be continuous and suppose $f(0) \neq 0$. Then the following
are equivalent: 
\begin{enumerate}
\item $f[A] \in \br_n^2$ for all $A \in \bp_n^1(I)$; 
\item $f(x) = a + b \phi_\alpha(x)$ or $f(x) = a + b \psi_\alpha(x)$ for
$a \neq 0$, $\alpha > 0$, and $b \in \R$. 
\end{enumerate}
\end{corollary}

Recall that the power functions $\phi_\alpha, \psi_\alpha$ were defined
in equation \eqref{eqn:phi_psi_alpha}.

\begin{proof}
Clearly $(2) \Rightarrow (1)$ by Lemma \ref{lem:continuous_rank1}. To
show the converse, apply the implication $(1) \Rightarrow (3)$ of
Proposition \ref{prop:general_rank_reduction}, with $h$ replaced by $f$
to obtain that $(f - f(0))[-]: \bp_{n-1}^1(I) \to \br_{n-1}^1$. The
result now follows by Lemma \ref{lem:continuous_rank1}.
\end{proof}

We now have all the ingredients needed to prove the first main result of
the paper. 

\begin{proof}[Proof of Theorem \ref{thm:sum_powers}]
We show the result for $I = (-R,R)$; the proof is similar for $I =
[0,R)$. 
The proof proceeds by building up the polynomial function $f(x)$ step by
step, in a way that is similar to Horner's algorithm. In order to do so,
given $r \in \Z_{\geq 0}$, define an operator $T_r$ mapping any function
$h: I \to \R$ admitting at least $r$ left and right derivatives at zero,
via:
\begin{equation}
T_r(h)(x) := \begin{cases}
\frac{ h(x) - \frac{h^{(r)}(0)}{r!} \cdot x^r}{x^r} & \textrm{if }
x\not=0 \\
0 & \textrm{if } x=0. 
\end{cases}
\end{equation}

\noindent Denote by $f^{(m_1)}, f^{(m_2)}, \dots, f^{(m_k)}$ the first
$k$ derivatives of $f$ that are nonzero at $0$, with $0 \leq m_1 < \cdots
< m_k$ (where $k$ is taken to be $n$ for part (3) of the statement).
Here we define $f^{(m)} = f$ when $m=0$.
Also define $m_0 := 0$ for notational convenience.
Now inductively construct the function $M_i(x)$ for $0 \leq i \leq k-1$
via $M_0(x) := f(x)$ and $M_i(x) := T_{m_i - m_{i-1}} M_{i-1}(x)$ for $1
\leq i \leq k-1$. We now claim that
\begin{equation}\label{eqn:taylor_Mi}
M_i(x) = \sum_{j=i+1}^{k-1} \frac{f^{(m_j)}(0)}{m_j!} x^{m_j-m_i} +
O(x^{1+m_{k-1}-m_i}), \qquad \forall x \in I, \ 0 \leq i \leq k-1.
\end{equation}

\noindent For $i=0$, the claim is easily verified using Taylor's theorem.
Now apply the operators $T_{m_j - m_{j-1}}$ inductively to verify the
claim for each $0 \leq i \leq k-1$.

It follows from \eqref{eqn:taylor_Mi} that $M_i$ is continuous at zero
for all $0 \leq i \leq k-1$. Next, we claim that for $i=0, \dots, k-1$
and $A \in \bp_{n-i}^1(I)$, we have $M_i[A] \in \br_{n-i}^{k-i}$. We will
prove the claim by induction on $i \geq 0$. Clearly the result holds if
$i=0$. Now assume it holds for some $i-1 \geq 0$. By definition, for $x
\ne 0$, 
\begin{align}\label{eqn:M_itaylor}
M_i(x) = x^{-(m_i-m_{i-1})} \left(M_{i-1}(x) - \frac{f^{(m_i)}(0)}{m_i!}
x^{m_i-m_{i-1}}\right) = \frac{M_{i-1}(x)}{x^{m_i-m_{i-1}} } -
\frac{f^{(m_i)}(0)}{m_i!}.
\end{align}

\noindent By Propositions \ref{prop:rank_onesided} and
\ref{prop:general_rank_reduction} and the induction hypothesis, it
follows that $M_i[A] \in \br_{n-(i-1)-1}^{k-(i-1)-1} = \br_{n-i}^{k-i}$
for all $A \in \bp_{n-i}^1(I)$. This completes the induction and proves
the claim for all $0 \leq i \leq k-1$. In particular, $M_{k-1}[A] \in
\br_{n-k+1}^1$ for all $A \in \bp_{n-k+1}^1(I)$. Moreover, $M_{k-1}$ is
continuous on $I$ by \eqref{eqn:taylor_Mi}. We now complete the proofs of
the three parts separately. \medskip

\noindent \textit{Proof of (1).}
By Lemma \ref{lem:continuous_rank1}, $M_{k-1}(x) = a \phi_\alpha(x)$ or
$a \psi_{\alpha}(x)$ for some $\alpha > 0$ and $a \in \R$. Working
backwards, it follows that $f(x) = P(x) + c \phi_\gamma(x)$ or $f(x) =
P(x) + c \psi_\gamma(x)$ where $P$ is a polynomial with exactly $k-1$
nonzero coefficients, $c \in \R$, and $\gamma \geq m_{k-1}$. Let $m_k$ be
the least positive integer such that $m_k > m_{k-1}$ and $f^{(m_k)}(0)$
exists and is nonzero. Then we obtain that $\gamma = m_k$, so that $f$ is
a polynomial with exactly $k$ nonzero coefficients (and hence exactly $k$
nonzero derivatives at zero). This proves the first part of the
theorem.\medskip

\noindent \textit{Proof of (2).}
To show the second part, we claim that $M_i[A] \in \bp_{n-i}^{k-i}$ for
all $A \in \bp_{n-i}^1(I)$ and $0 \leq i \leq k-1$. Indeed, the result
clearly holds for $i=0$. Now assume the result holds for $i-1 \geq 0$.
Note first that by applying Proposition \ref{prop:rank_onesided}(2) to
$f(x) = M_{i-1}(x)$, $g(x) = x^{m_i-m_{i-1}}$ and $c = f^{(m_i)}(0)$, we
obtain that $f^{(m_i)}(0) \geq 0$. Moreover, by \eqref{eqn:M_itaylor} and
Propositions \ref{prop:rank_onesided} and
\ref{prop:general_rank_reduction}, it follows that $M_i[A] \in
\bp_{n-i}^{k-i}$ for all $A \in \bp_{n-i}^1(I)$.
Using a similar argument as in part (1) together with Lemma
\ref{lem:continuous_rank1}, it follows that $f$ is a polynomial with
nonnegative coefficients and exactly $k$ nonzero coefficients. \medskip

\noindent \textit{Proof of (3).}
For the third part, we obtain that $h(x) := M_{k-1}(x) = M_{n-1}(x)$
satisfies $h[-]: \bp_n^1(I) \to \bp_n$. Therefore, $h$ maps $I \cap
[0,\infty)$ to $[0,\infty)$. Working backwards and reasoning as in the
previous parts, it follows that $f(x) = P(x) + x^{m_{n-1}} h(x)$ for a
polynomial $P(x)$ with exactly $k-1$ positive coefficients, and $h: I \to
\R$ such that $h(I \cap [0,\infty)) \subset [0, \infty)$. 
\end{proof}

\begin{remark}
Part (3) of Theorem \ref{thm:sum_powers} provides a necessary condition
for a function to map every rank $1$ $n \times n$ positive semidefinite
matrix with positive entries to an $n \times n$ positive semidefinite
matrix.
Note that even when $f(x) = \sum_{i=0}^N c_i x^{\alpha_i}$ for $\alpha_i
\geq 0$, condition (3) does not imply that all the coefficients $c_i$ are
nonnegative. Indeed, the function $h$ (in the statement of the theorem)
could be a sum of powers containing some negative coefficients, as long
as $h$ is nonnegative on $I$. It can however be shown that the first and
last $n$ coefficients of $f$ have to be positive (see \cite{Fischer92}
for more details).
\end{remark}

An important special case of interest in the literature is to study which
analytic functions preserve positivity. The following result
characterizes the analytic entrywise maps $f[-]$ sending $\bp_n^1(I)$ to
$\br_n^k, \bp_n^k$. Note that the third part of Theorem
\ref{thm:rank_one} generalizes Theorem \ref{thm:sum_powers}(3). 

\begin{theorem}[Rank $1$, fixed and arbitrary
dimension]\label{thm:rank_one}
Let $0 < R \leq \infty$, $I = [0, R)$ or $(-R,R)$, and let $f: I
\rightarrow \mathbb{R}$ be analytic on $I$. Also fix $1 \leq k < n$ in
$\N$.
\begin{enumerate}
\item Then $f[A] \in \br_n^k$ for all $A \in \bp_n^1(I)$ if and only if
$f$ is a polynomial with at most $k$ nonzero coefficients.

\item Similarly, $f[A] \in \bp_n^k$ for all $A \in \bp_n^1(I)$ if and
only if $f$ is a polynomial with at most $k$ nonzero coefficients, all of
which are positive.

\item Furthermore, $f[A] \in \bp_n$ for all $A \in \bp_n^1(I)$ and all $n
\in \N$, if and only if $f$ is absolutely monotonic on $I$. 
\end{enumerate}
\end{theorem}

\begin{proof}
First suppose $I = [0,R)$. The first two parts follow immediately from
Theorem \ref{thm:sum_powers} since $f$ is analytic (considering the cases
when $f$ has at least $k$ nonzero derivatives at the origin, and when it
does not).
Next, the sufficiency in the third part follows from the Schur product
theorem. To show the necessity, it suffices to show by standard results
from classical analysis (see Theorem \ref{thm:abs_monotonic_equiv}) that
$f(x) = \sum_{i=0}^\infty a_i x^i$ on $I$, with $a_i \geq 0$ for all $i$.
Now applying Theorem \ref{thm:sum_powers}(3), it follows that $f^{(i)}(0)
\geq 0$ for every $i \geq 0$; i.e., $f$ is absolutely monotonic on the
positive real axis. Finally, if $I = (-R,R)$, then the result follows
from the above analysis and the uniqueness principle for analytic
functions. 
\end{proof}

\subsection{Preserving positivity of rank 1 matrices and Laplace
transforms}\label{subsec:rank1}

We continue our study of rank constrained functions by exploring
functions mapping $\bp_n^1$ into $\bp_n$ for all $n \geq 1$. Such
functions can be characterized using the Laplace transform via the theory
of positive definite kernels, which we recall for the reader's
convenience.

\begin{definition}[{\cite[Chapter VI]{widder}}]
Let $I \subset \R$. A function $k : I \times I \rightarrow \R$ is a
\emph{positive definite kernel} on $I$ if for every finite sequence
$(x_i)_{i=1}^n \subset I$ of distinct numbers, the quadratic form 
\begin{equation}
Q(\xi) = \sum_{i=1}^n \sum_{j=1}^n k(x_i, x_j) \xi_i \xi_j \qquad (\xi
\in \mathbb{R}^n)
\end{equation}

\noindent is positive semidefinite. Equivalently, for every finite
sequence $(x_i)_{i=1}^n \subset I$ of distinct numbers, the matrix
$\left(k(x_i,x_j)\right)_{ij}$ is positive semidefinite. 
\end{definition}

Recall from classical results in analysis that positive definite kernels
can be characterized using the Laplace transform:

\begin{theorem}[{\cite[Chapter VI, Theorem 21]{widder}}]\label{thm:widder}
A function $f:(0,\infty) \rightarrow \mathbb{R}$ can be represented as
$f(x) = \int_{-\infty}^\infty e^{-\alpha x} d\mu(\alpha)$ for a positive
measure $\mu$ on $\mathbb{R}$ if and only if $f$ is continuous and the
kernel $k(x,y) := f(x+y)$ is positive definite on $(0,\infty)$. Moreover,
if $f$ can be written in the above form, then $f$ is analytic on
$(0,\infty)$.
\end{theorem}

Using Theorem \ref{thm:widder}, we now easily obtain the following
characterization of entrywise functions defined on $(0,R)$, which map
$\bp_n^1((0,R))$ into $\bp_n$ for every $n \geq 1$.

\begin{theorem}\label{thm:rank1}
Given $0 < R \leq \infty$ and $f: (0,R) \rightarrow \mathbb{R}$, the
following are equivalent: 
\begin{enumerate}
\item $f$ is continuous on $(0,R)$ and $f[A] \in \bp_n$ for all $A \in
\bp_n^1((0,R))$ and all $n$;

\item $f$ is continuous on $(0,R)$ and the kernel $k(x,y) =
f(e^{-(x+y-\ln R)})$ is positive definite on $(0,\infty)$; 

\item $g(x) := f(e^{-x})$ is continuous on $(-\ln R, \infty)$ and the
kernel $k(x,y) := g(-\ln R + x+y)$ is positive definite on $(0,\infty)$;

\item There exists a positive measure $\mu$ such that the function $g(x)
:= f(e^{-x})$ can be represented as 
\begin{equation}
g(x) = \int_{-\infty}^\infty e^{- \alpha x} d\mu(\alpha) \qquad (x > -
\ln R); 
\end{equation}

\item There exists a positive measure $\mu$ such that 
\begin{equation}
f(x) = \int_{-\infty}^\infty x^{\alpha} d\mu(\alpha) \qquad (0 < x < R); 
\end{equation}
\end{enumerate}

\noindent In particular, $g(x) := f(e^{-x})$ is analytic on $(-\ln R,
\infty)$.
\end{theorem}

\begin{proof}
Clearly, (5) $\Leftrightarrow$ (4). Now suppose (4) is true and consider
the function $h(x) := g(x-\ln R)$ for $x > 0$. Then 
\begin{equation*}
h(x) = \int_{-\infty}^\infty R^\alpha e^{- \alpha x} d\mu(\alpha) =
\int_{-\infty}^\infty  e^{- \alpha x} d\nu(\alpha) \qquad (x > 0), 
\end{equation*}

\noindent where $d\nu(\alpha) = R^\alpha d\mu(\alpha)$. By Theorem
\ref{thm:widder}, the kernel $h(x+y) = g(x+y-\ln R)$
is positive definite. This proves (3).  Conversely, if (3) is
true, then there exists a positive measure $\nu$ such that 
\begin{equation*}
h(x) = \int_{-\infty}^\infty  e^{- \alpha x} d\nu(\alpha) \qquad (x > 0). 
\end{equation*}

\noindent It follows that 
\begin{equation*}
g(x) = \int_{-\infty}^\infty  R^{-\alpha} e^{- \alpha x} d\nu(\alpha) =
\int_{-\infty}^\infty  e^{- \alpha x} d\mu(\alpha) \qquad (x > -\ln R), 
\end{equation*}

\noindent where $d\mu(\alpha) = R^{-\alpha} d\nu(\alpha)$. This proves
$(4) \Leftrightarrow (3)$. Clearly, $(3) \Leftrightarrow (2)$. Finally,
(2) is true if and only if the matrix 
$\left(f(R e^{-x_i} e^{-x_j})\right)_{ij}$
is positive semidefinite for every $x_1, \dots, x_n \geq 0$.
Equivalently, $f[A]$ is positive semidefinite for every positive
semidefinite matrix $A$ of rank $1$ with entries on $(0,R)$. Therefore,
$(2) \Leftrightarrow (1)$.
Finally, that $g$ is analytic also follows from Theorem \ref{thm:widder}.
\end{proof}

Recall that part (3) of Theorem \ref{thm:rank_one} provides a direct
characterization from first principles of the analytic maps sending
$\bp_n^1(I)$ to $\bp_n$. The same result can also be obtained using deep
results about the representability of functions as the Laplace transforms
of positive measures, and the uniqueness principle for the Laplace
transform (see \cite[Chapter VI, Theorem 6a]{widder}).

\begin{theorem}\label{thm:rank1_analytic}
Let $0 < R \leq \infty$, $I = [0, R)$ or $(-R,R)$, and let $f: I
\rightarrow \R$ be analytic on $I$. Then
$f[-] : \bp_n^1(I) \to \bp_n$ for all $n$
if and only if $f$ is absolutely monotonic on $I$. 
\end{theorem}

\begin{proof}
Let $f(z) = \sum_{n=0}^\infty a_n z^n$ for every $z$ in an open set in
$\mathbb{C}$ containing $I$. Then 
\begin{equation*}
g(x) := f(e^{-x})= \sum_{n=0}^\infty a_n e^{-nx} \qquad (x > -\ln R). 
\end{equation*}

\noindent Using this power series representation, the function $g$ can be
extended analytically to every $z$ in the half-plane $\{z \in \mathbb{C}
: \ree z > -\ln R\}$, i.e., $g(z) = \sum_{n=0}^\infty a_n e^{-nz}$
whenever $\ree z > -\ln R$.
Since $f[A]$ is positive semidefinite for every positive semidefinite
matrix $A$ of rank $1$ with coefficients in $(0,R)$ then, by Theorem
\ref{thm:rank1}, there exists a positive measure $\mu$ such that 
\begin{equation*}
g(x) = \int_{-\infty}^\infty e^{-\alpha x} d\mu(\alpha) \qquad (x > -\ln
R). 
\end{equation*}

\noindent The function 
\begin{equation*}
\widetilde{g}(z) = \int_{-\infty}^\infty e^{-\alpha z} d\mu(\alpha)
\qquad (\ree z > -\ln R)
\end{equation*}

\noindent provides an analytic extension of $g$ to $\{z \in \mathbb{C} :
\ree z > -\ln R\}$. Since $g$ and $\widetilde{g}$ are both analytic and
coincide on $x > -\ln R$, by the uniqueness principle for analytic
functions, $g(z) = \widetilde{g}(z)$ for every $z \in \{w \in \mathbb{C}
: \ree w > -\ln R\}$. Now, since $g$ and $\widetilde{g}$ are both
bilateral Laplace transforms and coincide in a common strip of
convergence, by \cite[Chapter VI, Theorem 6a]{widder}, the two
representing measures must coincide. In other words, $\mu =
\sum_{n=0}^\infty a_n \delta_n$, where $\delta_n$ denotes the Dirac
measure at the integer $n$. Since $\mu$ is positive, it follows that $a_n
\geq 0$ for every $n \geq 0$ and so $f$ is absolutely monotonic. The
converse follows immediately from the Schur product theorem.
\end{proof}

\subsection{Two-sided extensions of power functions}\label{S2sided}

Thus far in this section, we have worked mostly with polynomials as the
continuous functions sending $\bp_n^1(I)$ to $\bp_n^k$ for integers $1
\leq k < n$, and $I = [0,R)$ or $(-R,R)$ for some $0 < R \leq \infty$.
The strategy in proving all of the characterizations obtained above in
this section was to subtract the ``lowest degree monomial in $f$'' and
obtain a function that sends $\bp_{n-1}^1(I)$ to $\bp_{n-1}^{k-1}$. The
final step classified the continuous maps sending $\bp_n^1(I)$ to
$\br_n^1$, and these are precisely the constants and scalar multiples of
the maps $\phi_\alpha, \psi_\alpha$ for $\alpha > 0$.

In this subsection, we take a closer look at the above steps, but under
less restrictive assumptions. Our main result in this subsection
generalizes Theorem \ref{thm:sum_powers} under more relaxed
differentiability hypotheses on $f$. To prove this result, we adopt the
three-step approach from Section \ref{S3step}.

\begin{theorem}\label{thm:thmA_2sided}
Let $0 < R \leq \infty$, $I = (-R,R)$, and $f: I \rightarrow \mathbb{R}$.
Fix $1 \leq k < n$, and let $0 \leq m_1 < m_2 < \dots < m_k$ denote the
orders of the first nonzero left and right derivatives of $f$ at $0$.
Assume moreover that $|f^{(m_i)}(0^+)| = |f^{(m_i)}(0^-)|$ for $1 \leq i
\leq k$, that $f(0) = 0$, and either $f$ is continuous on $I$ or $k > 1$.
\begin{enumerate}
\item Then $f[-] : \bp_n^1(I) \to \br_n^k$ if and only if $f(x) =
\sum_{i=1}^k c_i g_{m_i}(x)$ where $c_i \in \R \setminus \{0\}$, $0 \leq
m_1 < m_2 < \dots < m_k$ are integers, and $g_{m_i}(x) = \phi_{m_i}(x)$
or $\psi_{m_i}(x)$. 
\item Similarly, $f[-] : \bp_n^1(I) \to \bp_n^k$ if and only if $f$ is of
the same form with all $c_i > 0$. 

\end{enumerate}
In particular, $f$ has exactly $k$ nonzero left and right derivatives at
$0$.
\end{theorem}

\begin{proof}
By Taylor's theorem, 
\begin{align}
& f(x) = \sum_{i=1}^k \frac{f^{(m_i)}(0^+)}{m_i!} x^{m_i} +o(x^{m_k})
\qquad (x \in (0,R)), \label{eqn:taylor_right} \\
& f(x) = \sum_{i=1}^k (-1)^{m_i}\frac{f^{(m_i)}(0^-)}{m_i!} x^{m_i} +
o(x^{m_k}) \qquad (x \in (-R, 0)). \label{eqn:taylor_left}
\end{align}

\noindent For $i=1, \dots, k$, define a function $g_{m_i}$ as follows: if
$f^{(m_i)}(0^+) = f^{(m_i)}(0^-)$, then $g_{m_i} = \phi_{m_i}$ if $m_i$
is even and $g_{m_i} = \psi_{m_i}$ if $m_i$ is odd.  If instead
$f^{(m_i)}(0^+) = -f^{(m_i)}(0^-)$, then $g_{m_i} = \psi_{m_i}$ if $m_i$
is even and $g_{m_i} = \phi_{m_i}$ if $m_i$ is odd. Using this notation,
equations \eqref{eqn:taylor_right} and \eqref{eqn:taylor_left} can be
rewritten as 
\begin{equation}\label{eqn:sum_g_i}
f(x) = \sum_{i=1}^k \frac{f^{(m_i)}(0^+)}{m_i!} g_{m_i}(x) +
o(x^{m_k})\qquad (x \in I). 
\end{equation}

\noindent Note that \eqref{eqn:sum_g_i} also holds at $x=0$ since $f(0) =
0$ by assumption. 

For $r \in \Z_{\geq 0}$, define an operator $T_r$ mapping any function
$h: I \to \R$ admitting at least $r$ left and right derivatives at zero,
via:
\begin{equation}
T_r(h)(x) := \begin{cases}
\frac{ h(x) - \frac{h^{(r)}(0^+)}{r!} \cdot g_r(x)}{g_r(x)} & \textrm{if }
x\not=0, \\
0 & \textrm{if } x=0,  
\end{cases}
\end{equation}
where 
\begin{equation}
g_r(x) := \begin{cases}
\phi_r(x) & \textrm{ if } f^{(m_i)}(0^+) = f^{(m_i)}(0^-) \textrm{ and }
m_i \textrm{ is even}, \\
\psi_r(x) & \textrm{ if } f^{(m_i)}(0^+) = f^{(m_i)}(0^-) \textrm{ and }
m_i \textrm{ is odd}, \\
\psi_r(x) & \textrm{ if } f^{(m_i)}(0^+) = - f^{(m_i)}(0^-) \textrm{ and
} m_i \textrm{ is even}, \\
\phi_r(x) & \textrm{ if }  f^{(m_i)}(0^+) = -f^{(m_i)}(0^-) \textrm{ and
} m_i \textrm{ is odd}. 
\end{cases}
\end{equation}

Now, inductively construct the function $M_i(x)$
for $0 \leq i \leq k-1$ via: $M_0(x) := f(x)$ and $M_i(x) := T_{m_i -
m_{i-1}} M_{i-1}(x)$ for $1 \leq i \leq k-1$. We have that 
\begin{equation}
M_i(x) = \sum_{j=i+1}^{k-1} \frac{f^{(m_j)}(0^+)}{m_j!} h_{m_j-m_i}(x) +
o(x^{1+m_{k-1}-m_i}), \qquad \forall x \in I, \ 0 \leq i \leq k-1, 
\end{equation}
where $h_{m_j-m_i}(x) = \phi_{m_j-m_i}(x)$ or $h_{m_j-m_i}(x) =
\psi_{m_j-m_i}(x)$. The rest of the proof is now similar to the proof of
Theorem \ref{thm:sum_powers}. Note that if $f$ is continuous or $k > 1$,
then $M_{k-1}(x)$ is continuous and sends $\bp_{n-k+1}^1(I)$ to
$\br_{n-k+1}^1$. Now apply Lemma \ref{lem:continuous_rank1} to conclude
the proof of (1). The second part of the theorem follows using an
argument similar to the one used in the proof of Theorem
\ref{thm:sum_powers}. 
\end{proof}

We now present an application that further illustrates the power of the
three-step approach described in this section. In order to do so, we
first extend the definition of the \textit{power functions} $\phi_\alpha,
\psi_\alpha$ introduced in \eqref{eqn:phi_psi_alpha} to also cover
negative powers as follows:
\begin{align}
\begin{aligned}
& \phi_\alpha(0) = \psi_\alpha(0) := 0, \qquad \phi_\alpha(x) :=
|x|^\alpha,\\
& \psi_\alpha(x) := \sgn(x) |x|^\alpha, \qquad
\forall \alpha \in \R, \ x \in \R \setminus \{ 0 \}.
\end{aligned}
\end{align}

\noindent It is easy to verify that for all $\alpha \in \R$, the power
maps $\phi_\alpha, \psi_\alpha$ are continuous except possibly at $0$, as
well as multiplicative. In fact, FitzGerald and Horn \cite{FitzHorn},
Bhatia and Elsner \cite{Bhatia-Elsner}, and Hiai \cite{Hiai2009} analyzed
the set of such power maps which preserve entrywise Loewner positivity
(as well as other properties such as monotonicity and convexity).
Recently in \cite{GKR-crit-2sided}, we have completed the classification
of these maps, which was initiated by FitzGerald and Horn in loc.~cit.

The following result generalizes a part of Theorem \ref{thm:rank_one} to
sums of (possibly non-integer) powers. 

\begin{theorem}\label{thm:sum_non_int_powers}
Fix $0 < R \leq \infty$, $I = [0, R)$ or $(-R,R)$, integers $n \geq 3$
and $1 \leq k < n$, and define the function
\begin{equation}
f(x) := c_0 + \sum_{j=1}^\infty c_j g_{\alpha_j}(x), \qquad x \in I,
\end{equation}

\noindent where $c_j, \alpha_j \in \mathbb{R}$ with $\alpha_j <
\alpha_{j+1}$ for $j \geq 1$, and $g_{\alpha_j} \equiv \phi_{\alpha_j}$
or $\psi_{\alpha_j}$ for all $j$. Assume that $f$ is continuous on $I
\setminus \{ 0 \}$. Then the following are equivalent: 
\begin{enumerate}
\item $f[A] \in \br_n^k$ for every $A \in \bp_n^1(I)$; 
\item $c_j \not= 0$ for at most $k$ values of $j$. 
\end{enumerate}
In particular, if $f[A] \in \bp_n^k$ for every $A \in \bp_n^1(I)$, then
$c_j \geq 0$ for all $j$. 
\end{theorem}

In order to carry out the first step of the three-step approach for sums
of two-sided powers, we need the following preliminary result, which is
analogous to Lemma \ref{lem:continuous_rank1} but without assuming
continuity at the origin.

\begin{lemma}\label{lem:cont}
Suppose $n \geq 3$, $0 < R \leq \infty$, $I = [0,R)$ or $(-R,R)$, and $f
: I \to \R$ is continuous except possibly at $0$. Then the following are
equivalent:
\begin{enumerate}
\item $f[-] : \bp_n^1(I) \to \br_n^1$.
\item There exists $c \in \R$ such that either $f \equiv c$ on $I$ or
$f(x) \equiv c \phi_\alpha(x)$ or $c \psi_\alpha(x)$ for some $\alpha \in
\R$.
\end{enumerate}

\noindent If instead $n = 2$, then $f[-] : \bp_2^1(I) \to \br_2^1$ if and
only if there exists $c \in \R$ such that either (2) holds, or $f \equiv
c$ on $I \cap [0,\infty)$ and $f \equiv -c$ on $I \cap (-\infty,0)$.
\end{lemma}

\begin{proof}
That $(2) \Rightarrow (1)$ (and the corresponding implication for $n=2$)
is clear. We now prove the converse implications. For ease of exposition,
we show this result in three steps. In Step $1$, we examine the behavior
of $f$ on $(0,R)$. In Step $2$, we study the possible values for $f$ on
$(-R,0)$ when $I = (-R,R)$. We conclude by showing in Step 3 that $f(0) =
0$.\medskip

\noindent \textit{Step 1.}
First note by Proposition \ref{prop:charac_rank1}(2) that if $f(a) = 0$
for some $a \in I \setminus \{ 0 \}$, then $f \equiv 0$ on $I$ and we are
done. Thus for the remainder of the proof, we will assume that $f$ is
nonzero on $(0,R)$ and $f[-] : \bp_n^1(I) \to \br_n^1$. Our next claim is
that there exist $c \neq 0$ and $\alpha \in \R$ such that $f(x) =
cx^\alpha$ on $(0,R)$.

To see why the claim holds, first note that $|f|[-]$ maps $\bp_n^1(I)$
into $\br_n^1$, which implies that $|f|[-] : \bp_2^1(I) \to \br_2^1$.
Now given $0 < a < b \in I \cap (0,\infty) = (0,R)$, one shows using
equation \eqref{Econt} that
\[
|f(\sqrt{ab} (b/a)^y)| = \sqrt{|f(a) f(b)|} \cdot \left|
\frac{f(b)}{f(a)} \right|^y, \qquad \forall y \in (-1/2,1/2).
\]

\noindent Define $\displaystyle \alpha := \frac{\ln |f|(b) - \ln
|f|(a)}{\ln(b) - \ln(a)}$. Then the previous equation yields:
\[
|f|(x) := c' x^\alpha\ \forall x \in (a,b) \subset I,
\qquad c' := \frac{\sqrt{|f(a) f(b)|}}{(ab)^{\alpha/2}} > 0.
\]

We now claim that $|f|(x) = c' x^\alpha$ for all $x \in (0,R)$. To see
this, first choose $k \in \N$ such that $\sqrt{ab} t_k \in (a,b)$, where
$t_k := \sqrt[k]{x/\sqrt{ab}}$. Then $x = t_k^k \sqrt{ab}$, and
$\sqrt{ab} t_k^m \in I$ for $0 \leq m \leq k+1$. Now define
\[
B_m := \begin{pmatrix} \sqrt{ab} t_k^m & \sqrt{ab} t_k^{m+1}\\ \sqrt{ab}
t_k^{m+1} & \sqrt{ab} t_k^{m+2} \end{pmatrix} \oplus {\bf 0}_{(n-2)
\times (n-2)} \in \bp_n^1(I), \qquad 0 \leq m \leq k-2.
\]

\noindent Since $f[B_0] \in \br_n^1$, we conclude by the above analysis
that
\[
|f|(\sqrt{ab} t_k^2) = \frac{|f|(\sqrt{ab} t_k)^2}{|f|(\sqrt{ab})} =
\frac{(c')^2 (\sqrt{ab} t_k)^{2\alpha}}{c' \sqrt{ab}^\alpha} = c'
(\sqrt{ab} t_k^2)^\alpha.
\]

\noindent Similar reasoning shows that $|f|(\sqrt{ab} t_k^{m+2}) = c'
(\sqrt{ab} t_k^{m+2})^\alpha$ whenever $0 \leq m \leq k-2$. In particular
by setting $m=k-2$, we obtain:
$|f|(x) = |f|(\sqrt{ab} t_k^k) = c' (\sqrt{ab} t_k^k)^\alpha = c'
x^\alpha$ for all $x \in (0,R)$. Now $f$ is continuous on $(0,R)$, whence
so is $f / |f| : (0,R) \to \{ \pm 1 \}$. Therefore $f/|f|$ is constant on
$(0,R)$, and we conclude that $f(x) = c x^\alpha$ for $x \in
(0,R)$ with $c \neq 0$.\medskip

\noindent \textit{Step 2.}
The previous step shows that $f(x) = x^\alpha$ for all $x \in (0,R)$. Now
assume $I = (-R,R)$. We claim that there exists a constant $\varepsilon
\in \{ \pm 1 \}$ such that $f(x) = \varepsilon c |x|\alpha$ for all $0 >
x \in I$. Indeed, given $0>x\in I$, applying $f$ entrywise to the matrix
$\begin{pmatrix} |x| & x\\ x & |x| \end{pmatrix} \oplus {\bf 0}_{(n-2)
\times (n-2)} \in \bp_n^1(I)$ shows that $f(x) = \pm c |x|^\alpha$. Once
again, $f/|f|$ is constant on $I \cap (-\infty,0)$, from which the claim
follows. \medskip

\noindent \textit{Step 3.} 
It remains to determine the value of $f(0)$. First suppose $\alpha \neq
0$. Then choose $x > 0$ such that $f(x) \neq f(0)$. Applying $f$
entrywise to the matrix $x {\bf 1}_{1 \times 1} \oplus {\bf 0}_{(n-1)
\times (n-1)} \in \bp_n^1(I)$ shows that $f(0) = 0$. Therefore $f \equiv
c \phi_\alpha$ or $c \psi_\alpha$ on $I$ for all $n \geq 2$, if $\alpha
\neq 0$.

Finally, suppose $\alpha = 0$. Applying $f$ entrywise to the matrix
$(R/2) {\bf 1}_{1 \times 1} \oplus {\bf 0}_{(n-1) \times (n-1)} \in
\bp_n^1(I)$ shows that $f(0) = 0$ or $f(0) = c$. This proves the
assertion (2) in all cases except when $n>2$, $I \cap (-\infty,0)$ is
nonempty, and $f \equiv c \psi_\alpha$ on $I \setminus \{ 0 \}$ with $c
\neq 0$. In this case, choose $x>0$ such that $\pm x \in I$, and apply
$f$ entrywise to the matrix
\[
A := \begin{pmatrix} x & -x & 0\\ -x & x & 0\\ 0 & 0 & 0 \end{pmatrix}
\oplus {\bf 0}_{(n-3) \times (n-3)} \in \bp_n^1(I).
\]

\noindent It follows that $f[A]$ has rank one, whence its leading
principal $3 \times 3$ minor must vanish. Now this minor equals $-4c
f(0)^2 = 0$, whence it follows that $f(0) = 0$. This concludes the proof.
\end{proof}

In order to prove Theorem \ref{thm:sum_non_int_powers}, we also need to
extend classical results about Vandermonde determinants to the odd and
even extensions of the power functions.

\begin{proposition}\label{Tvandermonde}
Fix $0 < R \leq \infty$ and $I = [0,R)$ or $(-R,R)$.
\begin{enumerate}
\item The functions $\{ \phi_\alpha, \psi_\alpha : \alpha \in \R \} \cup
\{ f \equiv 1 \}$ are linearly independent on $I = (-R,R)$, while on $I =
[0,R)$ the functions $\{ \phi_\alpha = \psi_\alpha = x^\alpha : \alpha
\in \R \} \cup \{ f \equiv 1 \}$ are linearly independent.

\item The functions in the previous part are also ``countably linearly
independent''. More precisely, suppose $f(x) = c_0 + \sum_{i=1}^\infty
(c_i \phi_{\alpha_i}(x) + d_i \psi_{\alpha_i}(x))$ with
$(c_i)_{i \geq 0}, (d_i)_{i \geq 1} \subset \R$ and $\alpha_1 < \alpha_2
< \dots$ a sequence of distinct real powers. If $f$ is convergent on $I$,
then $f \equiv 0$ on $I$ if and only if $c_i = d_i= 0$ for all $i$. 

\item Suppose $f(x)$ is a linear combination of at most $n$ functions
$\phi_\alpha, \psi_\alpha, 1$ on $(-R,R)$, or of the functions $x^\alpha,
1$ on $[0,R)$. Then $f[-]$ sends $\bp_n^1(I)$ to $\bp_n$ if and only if
all coefficients in $f$ are nonnegative.
\end{enumerate}
\end{proposition}

We also recall the following well-known fact, which is repeatedly used
below.

\begin{proposition}[{\cite[Chapter XIII, \S8, Example
1]{Gantmacher_Vol2}}]\label{Pgantmacher}
Given real numbers $\alpha_1 < \cdots < \alpha_n$ and $0 < x_1 < \cdots <
x_n$ for some $n \geq 1$, the generalized Vandermonde matrix
$(x_i^{\alpha_j})$ is totally positive. In other words, every square
minor has a positive determinant.
\end{proposition}

\begin{proof}[Proof of Proposition~\ref{Tvandermonde}]\hfill\medskip

\noindent \textit{Proof of (1).}
This part can be proved using Proposition \ref{Pgantmacher} or using a
generalization of the Dedekind Independence Theorem to arbitrary
semigroups; see \cite[Chapter II, Theorem 12]{artin_galois}.\medskip

\noindent \textit{Proof of (2).}
First suppose that $f(x) = c_0 + \sum_{i =1}^\infty c_i
\phi_{\alpha_i}(x)$ is convergent and identically zero on $I = [0,R)$.
Then $c_0 = f(0) = 0$, and moreover,
$c_1 = \lim_{x \to 0^+} x^{-\alpha_1} f(x)$ $= \lim_{x \to 0^+} 0 = 0$.
Similarly one shows inductively that $c_i = 0$ for all $i$.
Now suppose $f(x) = c_0 + \sum_{i =1}^\infty (c_i \phi_{\alpha_i}(x) +
d_i \psi_{\alpha_i}(x))$ is convergent and identically zero on $I =
(-R,R)$. Once again, $c_0 = f(0) = 0$, and by considering $f$ on $[0,R)$,
we obtain that $c_i + d_i = 0$ for all $i$. Similarly, considering $g(x)
:= f(-x)$ on $[0,R)$, we obtain that $c_i - d_i = 0$ for all $i$. It
follows that $c_i = d_i = 0$, which concludes the proof.\medskip

\noindent \textit{Proof of (3).}
Fix $f(x) = \sum_{i=1}^l c_i f_i(x)$, where $l \leq n$ and each $f_i(x)$
is of the form:
\begin{itemize}
\item $\phi_\alpha(x)$ or $\psi_\alpha(x)$ or $1$, for some $\alpha \in
\R$, if $I = (-R,R)$;
\item $x^\alpha$ or $1$ for some $\alpha \in \R$, if $I = [0,R)$.
\end{itemize}

\noindent It is easy to show that every $f(x)$ of the above form with all
$c_i \geq 0$ sends $\bp_n^1(I)$ to $\bp_n^k$. To show the converse,
assume first that $I = [0,R)$ and $f(x) = \sum_{i=1}^l c_i x^{\alpha_i}$
with all $c_i \neq 0$. Now choose any $0 < x_1 < \cdots < x_l < \sqrt{R}$
and define ${\bf x} := (x_1, \dots, x_l, {\bf 0}_{n-l})^T \in \R^n$; then
the $l$ vectors ${\bf x}^{\circ \alpha_i} := (x_1^{\alpha_i}, \dots,
x_l^{\alpha_i}, {\bf 0}_{n-l})^T$ are linearly independent by Proposition
\ref{Pgantmacher}. Therefore for every $j$, there exists a vector
$\beta_j \in \mathbb{R}^n$ such that $\langle \beta_j, {\bf x}^{\circ
\alpha_i}\rangle = \delta_{ij}$. Since $f[{\bf x} {\bf x}^T] \in
\bp_n^k$, it follows that $\beta_j^T f[{\bf x} {\bf x}^T] \beta_j = c_j
\geq 0$. A similar construction can be carried out when $f(x) = c_0 +
\sum_{i=1}^{l-1} c_i x^{\alpha_i}$. This proves the result for $I =
[0,R)$.

Finally, we show the result for $I = (-R,R)$. Given $f = \sum_{i=1}^l c_i
f_i$ as above, let $S := \{ \alpha \in \R : f_i = \phi_\alpha$ or
$\psi_\alpha \}$. Then $1 + 2|S| \geq n \geq l$. Now fix $0 < x_1 < \dots
< x_{|S|} < \sqrt{R}$; then the matrix
\begin{equation}
\Psi({\bf x},{\bf \alpha}) := \begin{pmatrix}
V & V & {\bf 0}_{|S| \times 1}\\
V & -V & {\bf 0}_{|S| \times 1}\\
{\bf 1}_{1 \times |S|} & {\bf 1}_{1 \times |S|} & 1
\end{pmatrix}
\end{equation}

\noindent is nonsingular. Here, $V$ is the generalized Vandermonde matrix
\[
V = (\phi_{\alpha_i}(x_j))_{i,j =1}^{|S|} = (\phi_{\alpha_i}(-x_j))_{i,j
=1}^{|S|}  = (\psi_{\alpha_i}(x_j))_{i,j =1}^{|S|},
\]

\noindent and it is nonsingular by Proposition \ref{Pgantmacher}. Now
consider the $l \times (1 + 2|S|)$ submatrix $B$ formed by the $l$ rows
of $\Psi({\bf x}, {\bf \alpha})$ which correspond to the nonzero
coefficients $c_i$. Then $B$ has column rank $l$. Choose an $l \times l$
submatrix $B'$ of $B$ that is nonsingular. The result now follows by
proceeding as in the $I = [0,R)$ case.
\end{proof}

We now have all the ingredients to prove Theorem
\ref{thm:sum_non_int_powers}. 

\begin{proof}[Proof of Theorem \ref{thm:sum_non_int_powers}]
Clearly $(2) \Rightarrow (1)$. Now assume $(1)$ holds. If $c_j$ is
nonzero for at most $k-1$ values of $j \geq 0$ then we are done; thus
suppose that $c_j$ is nonzero for at least $k$ values of $j$. To simplify
the proof, we will assume $c_0, \dots, c_{k-1} \not=0$; the proof of the
general case is similar. Using the three-step approach as in Section
\ref{S3step}, $(f - f(0))[-] : \bp_{n-1}^1(I) \to \br_{n-1}^{k-1}$ by
Proposition \ref{prop:general_rank_reduction}. Now applying Proposition
\ref{prop:rank_onesided}(1) with $g(x) := g_{\alpha_1}(x), c := c_1$, we
conclude that $f_1(x) := c_1 + \sum_{j=2}^\infty c_j
h_{\alpha_j-\alpha_1}(x)$ satisfies $f_1[-]: \bp_n^1(I) \to
\br_{n-1}^{k-1}$, where $h_{\alpha_j-\alpha_1}(x) :=
g_{\alpha_i}(x)/g_{\alpha_1}(x)$ is of the form
$\phi_{\alpha_j-\alpha_1}(x)$ or $\psi_{\alpha_j-\alpha_i}(x)$.
Continuing inductively in this manner, we arrive at $f_{k-1} : I \to \R$,
of the form
$f_{k-1}(x) = c_{k-1} + \sum_{j = k}^\infty c_j {\widetilde h}_{\alpha_j
- \alpha_{k-1}}(x)$
where $\widetilde{h}_{\alpha_j - \alpha_{k-1}}(x) = \phi_{\alpha_j -
\alpha_{k-1}}(x)$ or $\psi_{\alpha_j - \alpha_{k-1}}(x)$, $c_{k-1} \neq
0$, and $f_{k-1}[-] : \bp_{n-k+1}^1(I) \to \br_{n-k+1}^1$. Moreover,
$f_{k-1}$ is continuous on $I \setminus \{ 0 \} = (0,R)$ by construction.
There are now two cases:
\begin{enumerate}
\item The first case is when $k < n-1$. Then $n-k+1 \geq 3$, so by Lemma
\ref{lem:cont}, $f_{k-1}$ is either a constant or a scalar multiple of
$\phi_\alpha$ or $\psi_\alpha$ for some $\alpha \in \R$. Evaluating at
the origin shows that $f_{k-1} \equiv c_{k-1}$. Now applying Proposition
\ref{Tvandermonde} shows that $c_j = 0$ for all $j \geq k$. This
concludes the proof of the first equivalence.

\item The other case is when $k=n-1$, i.e., $n-k+1=2$. Then by Lemma
\ref{lem:cont}, either $f_{k-1}$ is a constant or a scalar multiple of
$\phi_\alpha$ or $\psi_\alpha$ for some  $\alpha \in \R$ (in which case
the same reasoning as in the previous case yields the result), or else $I
= (-R,R)$ and $f_{k-1} \equiv K_c$ on $I$, where $K_c(x) \equiv c \neq 0$
on $[0,R)$ and $K_c \equiv -c$ on $(-R,0)$. We now show that this latter
possibility cannot occur. Indeed, suppose by contradiction that
\[
f_{k-1}(x) = c_{k-1} + \sum_{j = k}^\infty c_j {\widetilde h}_{\alpha_j -
\alpha_{k-1}}(x) \quad \equiv \quad K_c(x) \ (c \neq 0).
\]

\noindent Evaluating both sides at zero yields: $c_{k-1} = c$, so that
when restricted to $[0,R)$, we obtain
\[
\sum_{j = k}^\infty c_j {\widetilde h}_{\alpha_j - \alpha_{k-1}}(x)
\equiv 0, \qquad x \in [0,R).
\]

\noindent Using Proposition \ref{Tvandermonde}(2) on $[0,R)$, we conclude
that $c_j = 0$ for all $j \geq k$, so that $f_{k-1} \equiv c_{k-1}$ on
$I$, which contradicts our assumption that $f_{k-1} \equiv K_c$ on $I$.
\end{enumerate}

\noindent The final assertion is shown similarly using Proposition
\ref{prop:rank_onesided}(2) instead of Proposition
\ref{prop:rank_onesided}(1). 
\end{proof}

\section{Preserving positivity under rank constraints II:\\
The special rank $2$ case}\label{Srank2}

Recall that in Section \ref{Srank1}, we had studied functions mapping rank
$1$ matrices into $\bp_n^k$. We now study the entrywise functions mapping
$\bp_n^2$ to $\bp_n^k$. More precisely, we study functions which preserve
positivity on a class of \textbf{special rank 2 matrices}, i.e., matrices
of the form
\begin{equation}\label{Especialrank2}
a {\bf 1}_{n \times n} + u u^T, \qquad a \in \R,\ u \in \R^n. 
\end{equation}

\noindent (We abuse notation slightly here, as the matrix $a {\bf 1}_{n
\times n} + u u^T$ is of rank at most 1 if $a=0$.) As we demonstrate in this
section, preserving positivity on these special rank 2 matrices greatly
constrains the possible entrywise functions. We begin by generalizing a
previous result by Horn \cite[Theorem 1.2]{Horn} (attributed to Loewner),
which provides a necessary condition for an entrywise function to
preserve positivity on special rank $2$ matrices. To our knowledge Horn's
result is the only known result in the literature involving entrywise
functions preserving positivity for matrices of a fixed dimension.

\begin{theorem}[Necessary conditions, fixed dimension]\label{Thorn}
Suppose $0 < R \leq \infty$, $I = (0,R)$, and $f : I \to \R$. Fix $2 \leq
n \in \N$ and suppose that $f[A] \in \bp_n$ for all $A \in \bp_n^2(I)$ of
the form $A = a {\bf 1}_{n \times n} + u u^T$, with $a \in [0,R), u \in
[0,\sqrt{R-a})^n$. Then $f \in C^{n-3}(I)$,
\[
f^{(k)}(x) \geq 0, \qquad \forall x \in I,\ 0 \leq k \leq n-3,
\]

\noindent and $f^{(n-3)}$ is a convex nondecreasing function on $I$. In
particular, if $f \in C^{n-1}(I)$, then $f^{(k)}(x) \geq 0$ for all $x
\in I, 0 \leq k \leq n-1$.
\end{theorem}

\begin{remark}
Note that Theorem \ref{Thorn} generalizes \cite[Theorem 1.2]{Horn} by
weakening the hypotheses in the following three ways: (1) $f$ is no
longer assumed to be continuous; (2) $f$ is assumed to preserve Loewner
positivity on a far smaller subset of matrices in $\bp_n^2(I)$; (3) the
entries of the matrices can come from $(0,R)$ instead of $(0,\infty)$,
for any $0 < R \leq \infty$.
\end{remark}

\begin{proof}[Proof of Theorem \ref{Thorn}]
For the sake of exposition, we carry out the proof in three
steps.\medskip

\noindent \textit{Step 1: Smooth case.}
First suppose that $f \in C^\infty(I)$ is smooth on $I = (0,R)$. The
result is then true for all $0 < R \leq \infty$, by repeating the
argument in the proof of \cite[Theorem 1.2]{Horn} on $I$, but using
$0<a<R$ now.\medskip

\noindent \textit{Step 2: Continuous case.}
Next, suppose $f$ is continuous but not necessarily smooth on $I =
(0,R)$. Given any probability distribution $\theta \in C^\infty(-1,0)$
with compact support in $(-1,0)$, let $\theta_\varepsilon(x) := \theta(x
\varepsilon^{-1})$ for $\varepsilon > 0$. Consider the function
$f_\varepsilon : (0,R-\varepsilon) \to \R$, given by
\begin{equation}\label{Ehorn}
f_\varepsilon(x) := \frac{1}{\varepsilon} \int_{-\varepsilon}^0 f(x-t)
\theta_\varepsilon(t)\ dt \in C^\infty(0,R-\varepsilon).
\end{equation}

\noindent Fix $0 < \varepsilon_0 < R$, and choose $A = a {\bf 1}_{n
\times n} + u u^T\in \bp_n(0,R - \varepsilon_0)$ with $a \in [0,R -
\varepsilon_0)$. Then equation \eqref{Ehorn} shows that for $0 <
\varepsilon \leq \varepsilon_0$,
\[
f_\varepsilon[A] = \int_{-\varepsilon}^0 \theta_\varepsilon(t) f[A - t
{\bf 1}_{n \times n}]\ dt \in \bp_n
\]

\noindent by assumption on $f$. Then $f_\varepsilon$ is smooth and
satisfies the other assumptions of the theorem on $I = (0,R -
\varepsilon_0)$. Therefore, by the previous step, all of the derivatives
of $f_\varepsilon$ are nonnegative on $(0, R - \varepsilon_0)$. In
particular, all the finite differences of $f_\varepsilon$ are
nonnegative. Since the finite differences of $f_\varepsilon$ converge to
the finite differences of $f$ as $\varepsilon \to 0^+$, it follows that
the finite differences of $f$ are also nonnegative on $ (0,R -
\varepsilon_0)$. Hence by \cite[Theorem, p.~497]{Boas-Widder}, $f \in
C^{n-3}(0,R-\varepsilon_0)$. The result now follows for $I = (0,R -
\varepsilon_0)$ by carrying out the steps at the end of the proof of
\cite[Theorem 1.2]{Horn}. (We remark that the continuity of $f$ is needed
in loc.~cit.) Finally, the result holds on all of $I = (0,R)$ because
$\varepsilon_0$ was arbitrary.\medskip

\noindent \textit{Step 3: General case.}
It remains to show that every function $f$ satisfying the hypotheses is
necessarily continuous on $I = (0,R)$. Consider any $a,b,c \in I$ such
that
\[
B := \begin{pmatrix} a & b\\b & c\end{pmatrix} \in \bp_2(I).
\]

\noindent We first claim that there exists $\widetilde{B} \in \bp_n(I)$
of the form $a' {\bf 1}_{n \times n} + u u^T$, whose principal $2 \times
2$ submatrix is $B$. To show the claim, define
\[
M_1 := \begin{pmatrix} a {\bf 1}_{(n-1) \times (n-1)} & b {\bf 1}_{(n-1)
\times 1}\\ b {\bf 1}_{1 \times (n-1)} & b^2 a^{-1}\end{pmatrix},
\qquad M_2 := (c - b^2 a^{-1}) E_{n,n},
\]

\noindent where $E_{n,n}$ is the elementary matrix, with $(i,j)$th entry
equal to $1$ if $i=j=n$ and $0$ otherwise. Now let $\widetilde{B} := M_1
+ M_2$. Note that $\widetilde{B} \in \bp_n^2(I)$ since $M_1,M_2 \in
\bp_n^1$.
Moreover, $\widetilde{B}$ contains $B$ as a principal submatrix. We now
show that $\widetilde{B}$ is indeed of the form $a' {\bf 1}_{n \times n}
+ u u^T$. There are two sub-cases:
\begin{enumerate}
\item If $a+c \leq 2b$, then $ac \geq b^2 \geq (a+c)^2/4$, whence $a=c$
by the arithmetic mean-geometric mean inequality. It follows that $b =
(a+c)/2 = a$, so that $\widetilde{B} = a {\bf 1}_{n \times n}$ is of the
desired form.

\item If $a+c > 2b$, set $a' := \frac{ac-b^2}{a+c-2b}$. It is easy to
verify that $0 \leq a' < \min(a,c)$. Therefore $\widetilde{B} = a' {\bf
1}_{n \times n} + u u^T$ is of the desired form, where $u :=
(\sqrt{a-a'}, \dots, \sqrt{a-a'}, \sqrt{c-a'})^T$.
\end{enumerate}

Finally, suppose $f[A] \in \bp_n$ for all $A \in \bp_n^2(I)$ of the form
$a' {\bf 1}_{n \times n} + u u^T$. Setting $A = \widetilde{B}$ for $B \in
\bp_2(I)$, we conclude that $f[B] \in \bp_2$ for all $B \in \bp_2(I)$. By
Theorem \ref{thm:vasudeva:M2}, $f$ is continuous on $I$, and the proof is
now complete.
\end{proof}

\begin{remark}
An immediate consequence of Theorem \ref{Thorn} is that for all
noninteger values $t \in (0,n-2)$, there exists $A \in \bp_n^2(I)$ of
the form $a' {\bf 1}_{n \times n} + u u^T$, such that $A^{\circ t} :=
((a_{ij}^t))$ is not in $\bp_n$. This strengthens \cite[Corollary
1.3]{Horn}. A specific example of such a matrix $A$ was constructed in
\cite[Theorem 2.2]{FitzHorn}. More generally for any $2 \leq l \leq n$,
one can produce examples of matrices $A \in \bp_n(I)$ of rank exactly $l$
such that $A^{\circ t} \not\in \bp_n$; see \cite[Section
6]{GKR-crit-2sided} for more details. 
\end{remark}

\begin{remark}\label{Thorn_implies_vasudeva}
Note that applying Theorem \ref{Thorn} for all $n \in \N$ easily yields a
generalization of Theorem \ref{vasudeva_rank2} for any interval $I =
(0,R)$. Thus, Theorem \ref{Thorn} immediately implies Theorem
\ref{thm:rank2_AM}. Later in Section \ref{Sam}, we will provide an
alternate, elementary proof of Theorem \ref{thm:rank2_AM}.
\end{remark}

Note that Theorem \ref{thm:rank2_AM} follows immediately from Theorem
\ref{Thorn}. In Section \ref{Sam}, we also provide an intuitive proof of
Theorem \ref{thm:rank2_AM} that uses the rank techniques developed in
this paper to prove Theorems \ref{thm:sum_powers} and \ref{Tltok}.

Recall that Theorem \ref{thm:sum_powers} shows that functions mapping
$\bp_n^1(I)$ to $\bp_n^k$ under some differentiability assumptions were
polynomials of arbitrary degree. We now show that the rank $2$ situation
is far more restrictive than the rank $1$ case, and it requires no
assumptions on $f$ if $k \leq n-3$.

\begin{theorem}[Special rank $2$, fixed dimension]\label{T2tok}
Suppose $0 < R \leq \infty$, $I = [0,R)$ or $(-R,R)$, and $f \in C^k(I)$
for some $1 \leq k < n$. 
\begin{enumerate}
\item Then the following are equivalent: 
\begin{enumerate}
\item $f[a {\bf 1}_{n \times n} + u u^T] \in \br_n^k$ for all $a \in
I$ and all $u \in \R^n$ with $a + u_i u_j \in I$;
\item $f$ is a polynomial of degree at most $k-1$.
\end{enumerate}

\item Similarly, when $I = [0,R)$, we have $f[a {\bf 1}_{n \times n} + u
u^T] \in \bp_n^k$ for all $a \in [0,R)$ and all $u \in \R^n$ with $a +
u_i u_j \in I$, if and only if $f$ is a polynomial of degree at most
$k-1$ with nonnegative coefficients. Moreover if $k \leq n-3$, the
assumption that $f \in C^k(I)$ is not required.
\end{enumerate}
\end{theorem}

\begin{remark}
Note that part (2) of Theorem \ref{T2tok} is stated only for $I = [0,R)$
since $a {\bf 1}_{n \times n} + uu^T \not\in \bp_n^2$ in general if $a <
0$. When $I = (-R,R)$ and $f$ is analytic on $I$, Theorem \ref{T2tok}(2)
also holds for $I = (-R,R)$ and follows immediately by the uniqueness
principle from the $I = [0,R)$ case. The result also holds if $I =
(-R,R)$ and $f$ admits at least $k$ nonzero derivatives at the origin,
since in that case $f$ is a polynomial by Theorem \ref{thm:sum_powers}.  
\end{remark}

The following result is crucially used in the proof of Theorem
\ref{T2tok}, as well as in later sections.

\begin{proposition}\label{PThmA_at_a}
Let $a \in \R$, $n \geq 2$, $1 \leq k \leq n$, $0 < R \leq \infty$, $I =
(a-R,a+R)$, and $f: I \to \R$. Suppose $f$ admits at least $k$ nonzero
derivatives at $a$. Then there exists $u \in \R^n$ such that $a + u_i u_j
\in I$ and $f[a {\bf 1}_{n \times n} + u u^T]$ has rank at least $k$.
\end{proposition}

\begin{proof}
Suppose to the contrary that $f[a {\bf 1}_{n \times n} + u u^T]$ has rank
less than $k$ for all $u \in \R^n$ such that $a + u_i u_j \in I$. Define
$g: (-R,R) \to \R$ by $g(x) := f(a+x)$. By hypothesis, $g[-]:
\bp_n^1((-R,R)) \to \br_n^{k-1}$.
Moreover, $g$ admits at least $k$ nonzero derivatives at $0$. Thus, by
Theorem \ref{thm:sum_powers}(1), the function $g$ is a polynomial with
exactly $k-1$ nonzero coefficients, which is impossible. Therefore, there
exists $u \in (-R,R)^n$ such that $g[uu^T] = f[a{\bf 1}_{n \times n} + u
u^T]$ has rank at least $k$. 
\end{proof}

We now have all the ingredients to prove Theorem \ref{T2tok}. 

\begin{proof}[Proof of Theorem \ref{T2tok}]
We begin by proving the first set of equivalences.\medskip

\noindent $\bf{(a) \Rightarrow (b)}$
Clearly, $(b)$ holds if $f^{(k)} \equiv 0$ on $I$. Thus, assume there is
a point $a_{k} \in I$ such that $f^{(k)}(a_k) \ne 0$. By continuity,
there is an open interval $I_k \subset I$ such that $f^{(k)}$ has no
roots in $I_k$. It follows by repeatedly applying Rolle's Theorem that
$f^{(i)}$ has at most $k-i$ roots in $I_k$ for all $0 \leq i < k$. Now
pick any point $a_0 \in I_k$ which is not one of these finitely many
roots of $f^{(i)}$ for any $0 \leq i \leq k$, i.e., $f^{(0)}(a_0),
f^{(1)}(a_0), \dots f^{(k)}(a_0) \ne 0$.
Therefore, by Proposition \ref{PThmA_at_a}, there exists $A = a {\bf
1}_{n \times n} + u u^T \in \bp_n^2(I)$ such that $f[A]$ has rank at
least $k+1$. This is impossible by assumption. Thus $f^{(k)} \equiv 0$
and $f$ is a polynomial of degree at most $k-1$, proving $(b)$.
\smallskip

\noindent $\bf{(b) \Rightarrow (a)}$
Conversely, suppose $f(x) = \sum_{m=0}^{k-1} c_m x^m$. Then we compute
for $a \in I$ and $u \in \R^n$ such that $a + u_i u_j \in I$:
\begin{align}
f[a {\bf 1}_{n \times n} + u u^T]_{ij} = &\ \sum_{m=0}^{k-1} c_m (a + u_i
u_j)^m = \sum_{m=0}^{k-1} \sum_{l=0}^m c_m \binom{m}{l} a^{m-l} u_i^l
u_j^l \notag\\
= &\ \sum_{l=0}^{k-1} u_i^l u_j^l \sum_{m=l}^{k-1} c_m \binom{m}{l}
a^{m-l} = \sum_{l=0}^{k-1} u_i^l u_j^l d_l, \label{Especial2}
\end{align}

\noindent say. Therefore
$\displaystyle f[a {\bf 1}_{n \times n} + u u^T] = \sum_{l=0}^{k-1} d_l
u^{\circ l} (u^{\circ l})^T$,
where $u^{\circ l} := (u_1^l, u_2^l, \dots, u_n^l)^T$. In particular,
$f[a {\bf 1}_{n \times n} + u u^T]$ has rank at most $k$.

We now prove the second set of equivalences. Clearly if $f$ is a
polynomial of degree $\leq k-1$ with nonnegative coefficients, then $f[a
{\bf 1}_{n \times n} + u u^T] \in \bp_n^k$ for all $a \geq 0$ and $u \in
\R^n$ such that $a + u_i u_j \in [0,R)$, by the calculation in equation
\eqref{Especial2}. Conversely if (1) holds, the first set of equivalences
already shows that $f$ is a polynomial of degree $\leq k-1$.
That the coefficients of $f$ are nonnegative follows by Theorem
\ref{thm:rank_one}.
Finally, if $k \leq n-3$ then the condition that $f \in C^k(I)$ actually
follows by Theorem \ref{Thorn}, and hence does not need to be assumed.
\end{proof}

\begin{remark}\label{rem:weakT2tok}
Note that the implication $(a) \Rightarrow (b)$ in Theorem \ref{T2tok}
also holds under the weaker assumption that $f[a{\bf 1}_{n \times n} +
uu^T] \in \br_n^k$ for all $a \in I$ and $u \in
(-\epsilon(a),\epsilon(a))^n$ where $0 < \epsilon(a) < \sqrt{R-|a|}$.
This observation will be important later in proving Theorem \ref{Tltok}. 
\end{remark}

Recall by Theorem \ref{T2tok} that polynomials of degree at most $k-1$
take special rank 2 matrices to $\br_n^k$. We now show that this
behavior is not shared by arbitrary linear combinations of powers - for
instance, if there is even one noninteger power involved.

\begin{proposition}\label{P49}
Fix $0 < R \leq \infty$, integers $n \geq 2$ and $m \geq 1$, and suppose
$\alpha_1 < \cdots < \alpha_m \in \R$ with $\alpha_i \notin \{ 0, 1,
\dots, n-2 \}$ for at least one $i$. Define $f(x) = \sum_{i=1}^m c_i
x^{\alpha_i}$, with $c_i \ne 0$. Then there exist $a \in (0,R)$ and $u
\in (-\epsilon,\epsilon)^n$ where $\epsilon :=
\min(\sqrt{a},\sqrt{R-a})$, such that $f[a {\bf 1}_{n \times n} + uu^T]$
has full rank.
\end{proposition}

\begin{proof}
We first claim that there exists an open interval $(p,q) \subset (0,R)$
such that $f, f', \dots, f^{(n-1)}$ are all nonzero on $(p,q)$. Indeed,
let $I_0 := (0,R)$, and note that for any $x_1 < \cdots < x_m$ in $I_0$,
the matrix $((x_j^{\alpha_i}))_{i,j=1}^m$ is nonsingular by Proposition
\ref{Pgantmacher}. Hence there exists $j$ such that $f(x_j) = \sum_i c_i
x_j^{\alpha_i} \neq 0$. We conclude by continuity of $f$ that $f =
f^{(0)}$ is nonzero on a nonempty open interval $I_1 \subset I_0$.
Repeatedly applying the above arguments, we obtain a nested sequence of
nonempty open intervals on which all sufficiently low-degree derivatives
of $f$ are nonzero. This shows the existence of the interval $I_{n-1} =
(p,q)$. (We need at least one $\alpha_i$ to not lie in $\{ 0, \dots, n-2
\}$, otherwise $f^{(n-1)} \equiv 0$.)
Finally, fix $a := (p+q)/2$ and $\epsilon = (q-p)/2$. The result then
follows by Proposition \ref{PThmA_at_a}.
\end{proof}

Proposition \ref{PThmA_at_a} also has the following important
consequence, which will be useful later. Recall that $f_\alpha(x) :=
x^\alpha$ for $\alpha > 0$. 

\begin{corollary}\label{Cnonint}
Let $\alpha \in (0,\infty)$ and $n \geq 2$. For $u \in \R^n$, define $A_u
:= {\bf 1}_{n \times n} + u u^T$. Then the following are equivalent: 
\begin{enumerate}
\item There exists $u\in \R^n$ (in fact in $(-1,1)^n$) such that the
matrix $f_\alpha[A_u]$ is nonsingular.
\item Either $\alpha \in \N \cap [n-1,\infty)$ or $\alpha \not\in \N$. 
\end{enumerate} 
\end{corollary}

\begin{proof}
Clearly, if $\alpha \in \N$, then for any $u \in \R^n$ such that $1+u_i
u_j \in (0,\infty)$, the matrix $f_\alpha[A_u] = \sum_{k=0}^\alpha
\binom{\alpha}{k}f_k[u]f_k[u]^T$ has rank at most $1+\alpha$. Therefore
if (1) holds, then $\alpha \not\in \N \cap (0,n-1)$ and so $(1)
\Rightarrow (2)$. Conversely, suppose $\alpha \in \N \cap [n-1,\infty)$
or $\alpha \not\in \N$. Then the function $f(x) = x^\alpha$ admits at
least $n$ nonzero derivatives at $x = 1$. Thus, by Proposition
\ref{PThmA_at_a}, there exists $u \in \R^n$ such that $1 + u_i u_j \in
(0,2)$ for all $i,j$ and $f_\alpha[A_u]$ has full rank. This shows that
$(2) \Rightarrow (1)$ and concludes the proof. 
\end{proof}

\section{Preserving positivity under rank constraints III:\\
The higher rank case}\label{Srank2_2}

The goal of this section is to study entrywise functions mapping
$\bp_n^l$ to $\bp_n^k$ for general $1 \leq k,l \leq n$. The $l=1$ case
has been explored in Section \ref{Srank1}, so we assume throughout this
section that $l>1$. Note by the results shown in Section \ref{Srank2}
that $C^k$ functions sending special rank 2 matrices to $\br_n^k$
automatically have to be polynomials. Using the aforementioned results,
in Subsection \ref{Sltok} we prove Theorem \ref{Tltok}, which classifies
the entrywise maps $f$ which are $C^k$ and send $\bp_n^l$ to $\br_n^k$.
We then show in Subsections \ref{Sk<l} and \ref{Sk<2l} that the
assumptions on $f$ can be relaxed even further if the rank ``does not
double''. Namely, we obtain a complete classification of the entrywise
maps sending $\bp_n^l$ to $\br_n^k$ for the special regime where $0 \leq
k < \min(n,2l)$, under either continuity assumptions on $f$ or no
assumption at all.

\subsection{Proof of the main Theorem \ref{Tltok}}\label{Sltok}

Before proceeding to the proof of Theorem \ref{Tltok}, we need some
preliminary results. 

\begin{proposition}\label{Pgvm}
Fix a field $\F$ of characteristic zero, as well as $N,l \in \N$. Let
${\bf m}_i = (m_{ij}) \in \F^l$ be distinct vectors for $1 \leq i \leq
N$. Then:
\begin{enumerate}
\item For any $r_1, \dots, r_{l-1} \in \N$, there exists $\alpha =
(\alpha_1, \dots, \alpha_l) \in \N^l \subset \F^l$, such that
$\alpha_{i+1} > r_i \alpha_i$ for all $0<i<l$, and $\alpha^T {\bf m}_i$
are pairwise distinct.

\item Suppose ${\bf m}_i \in \Z_{\geq 0}^l$ are distinct for all $i$, and
$0 < v_1 < \dots < v_N \in \Q \subset \F$. Then there exists $\alpha \in
\N^l$, such that defining $u_j := {\bf v}^{\circ \alpha_j}$ for $1 \leq j
\leq l$, the vectors
\[
{\bf w}_i := u_1^{\circ m_{i1}} \circ \cdots \circ u_l^{\circ m_{il}}
\]
are $\F$-linearly independent for $1 \leq i \leq N$. If $\F = \R$ then
the result holds even if we assume that ${\bf m}_i \in [0,\infty)^l$ for
all $i$.
\end{enumerate}
\end{proposition}

\begin{remark}
Note that the second part generalizes the nonsingularity of generalized
Vandermonde determinants, but in $\Q$ (and hence, every field of
characteristic zero). This is because in the special case of $l=1$, we
can choose $\alpha = \alpha_1 = 1$. Moreover, we will show below that
Proposition \ref{Pgvm}(2) is in fact equivalent to the nonsingularity of
generalized Vandermonde determinants.
\end{remark}

\begin{proof}[Proof of Proposition \ref{Pgvm}]
We first claim that if $V$ is a vector space
over a field $\F$ of characteristic zero, and $C$ is any $\Q$-convex
subset of $V$, then the following are equivalent:
\begin{enumerate}
\item $C$ is contained in a proper subspace of $V$.
\item $C$ is contained in a finite union of proper subspaces of $V$.
\end{enumerate}

Clearly (1) implies (2). Conversely, suppose $C$ is not contained in any
proper subspace of $V$. We show that (2) also fails to hold, by induction
on the number $n$ of proper subspaces of $V$. This is clearly true for
$n=1$. Next, suppose $C$ is not contained in a finite union of $n-1$
proper subspaces of $V$ and let $V_1,\dots,V_n$ be proper subspaces of
$V$. Fix elements $v_1 \in V_1 \setminus \bigcup_{i>1} V_i$ and $v_2 \in
V_2 \setminus (V_1 \cup \bigcup_{i>2} V_i)$, and consider the infinite
set $\{ (1/n) v_1 + ((n-1)/n) v_2 : n \in \N \} \subset C$. If $C \subset
\bigcup_{i=1}^n V_i$, then at least two elements of this infinite set lie
in some $V_i$, in which case we obtain that $v_1, v_2 \in V_i$. This is
false by assumption. It therefore follows by induction that $C$ is not
contained in a finite union of proper subspaces of $V$, and so $(2)$
implies $(1)$.

We now show the first part. Given $r_i$ as above, define
\[
r_0 := 0, \qquad N := \prod_{j=0}^{l-1} (1 + r_j),
\]
\[
C := \Q^l \cap \times_{i=1}^l \left( N^{-1} r_i \prod_{j=0}^{i-1} (1 +
r_j), N^{-1} \prod_{j=0}^i (1 + r_j) \right) \subset \F^l,
\]

\noindent where $\times_{i=1}^n$ denotes an $n$-fold Cartesian product of
intervals. Now clearly, $C$ is $\Q$-convex. Moreover, it is easy to check
that if $C$ is contained in any $\F$-vector subspace $V_0 \subset \F^l$,
then the standard basis $\{ e_i : 1 \leq i \leq l \}$ is contained in
$V_0$, whence $V_0 = \F^l$. Hence by the previous part, $C$ is not
contained in a finite union of proper subspaces of $\F^l$. In particular,
$C$ is not contained in the orthogonal complements to the vectors ${\bf
m}_i - {\bf m}_j$ (for all $i \neq j$ in $\Q^l$). Take any point in $C$
that is not contained in the union of these orthogonal complements, and
rescale it by a sufficiently large integer $M \in \N$. This provides the
desired vector $\alpha \in \N^l$.

Now fix $0 < v_1 < \dots < v_N \in \Q \subset \F$ and let $\alpha \in
\N^l$ be as in the above analysis. Since $\alpha^T {\bf m}_i$ are
distinct, the generalized Vandermonde matrix $B:= (v_t^{\alpha^T {\bf
m}_i})_{t,i=1}^N$ is nonsingular by Proposition \ref{Pgantmacher}.
Now define $u_j := (v_1^{\alpha_j}, \dots, v_N^{\alpha_j})^T$ for $1 \leq
j \leq l$. Then the linearly independent columns of $B$ are precisely
${\bf w}_i$ as defined in the statement, which concludes the proof. The
assertion for $\F = \R$ is similarly proved.
\end{proof}

The next two results are technical tools which will be useful in proving
Theorem \ref{Tltok} when $I = (-R,R)$. 

\begin{lemma}\label{lem:increase_dim}
Let $0 < R \leq \infty$, $I = [0, R)$ or $(-R,R)$, and $f: I \rightarrow
\mathbb{R}$. Fix $n \geq 2$, $0 \leq k < n-1$,
and $1 \leq l \leq n$. Then the following are equivalent: 
\begin{enumerate}
\item $f[-]: \bp_n^l(I) \to \br_n^k$.
\item $f[-]: \bp_{n'}^l(I) \to \br_{n'}^k$ for all $n' \geq \max(k+2,l)$. 
\end{enumerate}
\end{lemma}

\begin{proof}
Clearly $(2) \Rightarrow (1)$. Suppose $(1)$ holds. If $n' \leq n$ then
clearly $f[-]: \bp_{n'}^l(I) \to \br_{n'}^k$. Now suppose $n' \geq
\max(n+1,k+2,l)$ and let $A' \in \bp_{n'}^l(I)$. Note that every $n
\times n$ principal submatrix of $A'$ belongs to $\bp_n^l(I)$. It follows
by Lemma \ref{lem:rank_minors} that all $(k+1) \times (k+1)$ and $(k+2)
\times (k+2)$ principal minors of $f[A']$ vanish. Again using Lemma
\ref{lem:rank_minors}, we conclude that $f[A'] \in \br_{n'}^k$.  
\end{proof}

\begin{corollary}\label{cor:sp_rank2_neg}
Let $0 < R \leq \infty$, $I = (-R,R)$, and $f: I \rightarrow \mathbb{R}$.
Fix $n \geq 2$ and $0 \leq k < n-1$. Suppose $f[-]: \bp_n^2(I) \to
\br_n^k$. Then $f[a{\bf 1}_{n \times n}+u_iu_j] \in \br_n^k$ for all $a <
0$ and all $u \in \R^n$ such that $\pm a + u_i u_j \in I \ \forall i,j$. 
\end{corollary}

\begin{proof}
Let $u \in \R^n$ such that $\pm a + u_iu_j \in (-R,R) \ \forall i,j$.
Define 
\begin{equation*}
\mathbf{x} := \begin{pmatrix}-\sqrt{|a|} {\bf 1}_{n \times 1} \\
\sqrt{|a|} {\bf 1}_{n \times 1}\end{pmatrix} \in \R^{2n}, \qquad
\mathbf{y} := \begin{pmatrix} u \\ u\end{pmatrix} \in \R^{2n}, 
\end{equation*} 

\noindent and consider the matrix
\begin{equation*}
A := \mathbf{x}\mathbf{x}^T + \mathbf{y}\mathbf{y}^T = \begin{pmatrix}
|a| {\bf 1}_{n \times n} + uu^T & a {\bf 1}_{n \times n} + uu^T \\
a {\bf 1}_{n \times n} + uu^T & |a| {\bf 1}_{n \times n} + uu^T
\end{pmatrix}  \in \bp_{2n}^2(I). 
\end{equation*}

\noindent By Lemma \ref{lem:increase_dim}, we have $f[-]: \bp_{2n}^2(I)
\to \br_{2n}^k$ and so $f[A] \in \br_{2n}^k$. Thus, by Lemma
\ref{lem:rank_minors}, all $(k+1) \times (k+1)$ minors of $f[A]$ vanish.
In particular, all $(k+1) \times (k+1)$ minors of the upper right block
of $f[A]$ vanish. Thus, $f[a {\bf 1}_{n \times n} + u u^T] \in \br_n^k$.  
\end{proof}

We now combine the analysis in this and previous sections to prove our
second main result.

\begin{proof}[Proof of Theorem \ref{Tltok}]
If $k=0$ then the result is immediate to prove, so we suppose henceforth
that $k \geq 1$.

We first prove that $(2) \Rightarrow (1)$. Let $i \geq 0$ and $A =
\sum_{j=1}^l u_j u_j^T \in \bp_n^l(I)$. Then 
\begin{equation}
A^{\circ i} = \sum_{m_1 + \dots + m_l = i} \binom{i}{m_1, \dots, m_l}
{\bf w}_{\bf m} {\bf w}_{\bf m}^T, \qquad {\bf w}_{\bf m} := u_1^{\circ
m_1} \circ \cdots \circ u_l^{\circ m_l},
\end{equation}

\noindent where $\displaystyle \binom{i}{m_1, \dots, m_l}$ denotes the
multinomial coefficient. Note that there are exactly $\binom{i+l-1}{l-1}$
terms in the previous summation. Therefore $\rk A^{\circ i} \leq
\binom{i+l-1}{l-1}$, and so $(1)$ easily follows from $(2)$. 

Conversely, suppose (1) holds. If $I = [0,R)$, then $f$ is a polynomial
of degree at most $k-1$ by Theorem \ref{T2tok}. Similarly, if $I =
(-R,R)$, then an application of Corollary \ref{cor:sp_rank2_neg} shows
that $f[a{\bf 1}_{n \times n}+uu^T] \in \br_n^k$ for all $a \in (-R,R)$
and all $u \in \R^n$ such that $\pm a + u_i u_j \in (-R,R)$. Thus $f$ is
also a polynomial of degree at most $k-1$ by Theorem \ref{T2tok} and
Remark \ref{rem:weakT2tok}. 

Now denote by ${\bf m}_1, \dots, {\bf m}_N$ the collection of vectors
$(\Z_{\geq 0} \cap [0,k-1])^l$. By Proposition \ref{Pgvm}
with all $r_i = 1$, there exists $\alpha \in \N^l$ with distinct
coordinates such that $\alpha^T {\bf m}_i$ are pairwise distinct for $1
\leq i \leq N$. Let $g_\alpha: [0,\infty) \to \R$ be defined by 
\begin{equation}\label{E52}
g_\alpha(x) := R \cdot \frac{x^{\alpha_1} + \dots +
x^{\alpha_l}}{R^{\alpha_1} + \dots + R^{\alpha_l}}. 
\end{equation}
Note that $g_\alpha[-]: \bp_n^1(I) \to \bp_n^l(I)$ and so $f
\circ g_\alpha [-]: \bp_n^1(I) \to \br_n^k$. Thus, by Theorem
\ref{thm:rank_one}, the polynomial $f \circ g_\alpha$ is a linear
combination of at most $k$ integer powers. On the other hand, 
writing $f(x) = \sum_{t=1}^r a_t x^{i_t}$ for some distinct
integers $i_t \in [0, k-1]$ with all $a_t$ nonzero,
by the choice of $\alpha$, the function $f \circ g_\alpha$ is a linear
combination of exactly $\sum_{t=1}^r \binom{i_t+l-1}{l-1}$ distinct
integer powers. Therefore (2) follows since the power functions $\{ x^n :
n \in \Z_{\geq 0} \} \cup \{f \equiv 1\}$ are linearly independent on
$[0,\infty)$ by Proposition~\ref{Tvandermonde}.

We now prove the second set of equivalences. Clearly if $f$ is a
polynomial with nonnegative coefficients which satisfies assertion (2) in
the theorem, then $f[-] : \bp_n^l(I) \to \bp_n^k$ by the calculation in
equation \eqref{Especial2} (and the Schur product theorem). Conversely if
(1) holds, the first set of equivalences already shows that $f$ is a
polynomial of degree $\leq k-1$ satisfying \eqref{eqn:sum_binom}. That
the coefficients of $f$ are nonnegative follows by Theorem
\ref{thm:rank_one}.
Finally, if $k \leq n-3$ then the condition that $f \in C^k(I)$ actually
follows by Theorem \ref{Thorn}, and hence does not need to be assumed.
\end{proof}

\begin{remark}
Note that if $l > 1$, Theorem \ref{Tltok} immediately provides a
constraint on the degree of a polynomial $p(x)$ mapping $\bp_n^l$ to
$\bp_n^k$. Indeed, the degree must satisfy $\binom{\deg(p) +l-1}{l-1}
\leq k$. On the other hand, the degree can be arbitrary when $l=1$, by
Theorem \ref{thm:sum_powers}.
\end{remark}

Recall that Theorem \ref{thm:sum_powers} shows that under appropriate
differentiability hypotheses, entrywise functions mapping $\bp_n^1(I)$
into $\br_n^k$ are precisely the set of polynomials with $k$ nonzero
coefficients. Similarly, Theorem \ref{Tltok} shows that an analytic
function maps $\bp_n^l(I)$ to $\br_n^k$ if and only if it satisfies
equation \eqref{eqn:sum_binom}. We now prove that the conclusion of the
theorems are optimal in the following precise sense. 

\begin{proposition}
Let $0 < R \leq \infty$, $I = [0, R)$ or $(-R,R)$, and $f: I \rightarrow
\mathbb{R}$. Fix $1 \leq k < n$. Then: 
\begin{enumerate}
\item If $f$ is a polynomial with $k$ nonzero coefficients, then there
exists a matrix $A \in \bp_n^1(I)$ such that $f[A]$ has rank exactly $k$. 
\item If $f(x) = \sum_{t=1}^r a_{i_t} x^{i_t}$ with $a_{i_t} \ne 0$, and
$i_t \in \Z_{\geq 0}$ satisfying equation \eqref{eqn:sum_binom}, then
there exists $A \in \bp_n^l(I)$ such that $f[A]$ has rank exactly
$\sum_{t=1}^r \binom{i_t+l-1}{l-1}$. 
\end{enumerate}
\end{proposition}

\begin{proof}
It suffices to show the result for $I = [0,R)$. To prove (1), let $f(x) =
\sum_{i=1}^k c_i x^{m_i}$ with $c_i \ne 0$ and $m_i \in \N$. Let $v \in
I^n$ be a vector with distinct components and let $A = v v^T \in
\bp_n^1(I)$. Clearly $\rk f[A] \leq k$ since $f[A]$ is a sum of $k$ rank
$1$ matrices. Now, by Proposition \ref{Pgantmacher}, the vectors
$v^{\circ m_1}, \dots, v^{\circ m_k}$ are linearly independent. Denote by
$U$ the $k \times n$ matrix whose columns are $v^{\circ m_1}, \dots,
v^{\circ m_k}$, and let $C$ be the $k \times k$ diagonal matrix with
diagonal entries $c_1, \dots, c_k$. Note that $f[A] = U^T C U$. Clearly
the matrices $U^T C$ and $U$ have rank $k$. Thus, by Sylvester's rank
inequality, $\rk f[A] = \rk U^TCU \geq \rk U^T C + \rk U - k = k$. It
follows that $\rk f[A] = k$. This proves (1). 

To prove (2), first note that by Proposition \ref{Pgvm}(2), there exist
$u_1, \dots, u_l \in I^n$ such that the vectors 
\begin{equation}\label{eqn:vectors_multinom}
\{ u_1^{\circ a_1} \circ \dots \circ u_l^{\circ a_l} : a_1, \dots, a_l
\in \Z_{\geq 0}, a_1+\dots+a_l = i_t, t=1, \dots, r\}
\end{equation}

\noindent are linearly independent. Note that there are $\sum_{t=1}^r
\binom{i_t+l-1}{l-1}$ such vectors. Define $A  := \sum_{i=1}^l u_i u_i^T
\in \bp_n^l(I)$. Expanding $f[A]$ using the multinomial theorem, we
obtain a linear combination of the vectors in
\eqref{eqn:vectors_multinom} with nonzero coefficients. Using the same
argument as in the first part, it now follows that $f[A]$ has rank
$\sum_{t=1}^r \binom{i_t+l-1}{l-1}$, as desired. 
\end{proof}

\subsection{The regime $1 \leq k < l$}\label{Sk<l}

Recall that the characterization obtained in Theorem \ref{Tltok} was
obtained under the assumption that $f \in C^k$. Surprisingly, this
assumption can be relaxed significantly if additional constraints are
known on $(l,k)$. In this Subsection and Subsection \ref{Sk<2l}, we study
the cases where $1 \leq k < l$ and $l \leq k < 2l$ respectively. We now
demonstrate that when $k < l$, \textit{no} assumption on $f$ is required
in order to obtain the conclusion of Theorem \ref{Tltok}. 

\begin{theorem}\label{Tk<l}
Suppose $0 < R \leq \infty$, $I = [0,R)$ or $(-R,R)$, and $f : I \to \R$
with $f \not\equiv 0$.
Fix integers $n \geq 3$ and $1\leq k < l \leq n$. Suppose $1 \leq k <
n-1$ when $I = (-R,R)$.
Then the following are equivalent:
\begin{enumerate}
\item $f[A] \in \br_n^k$ for every $A \in \br_n^l(I)$; 
\item $f[A] \in \br_n^k$ for every $A \in \bp_n^l(I)$;
\item $f \equiv c$ on $I$ for some $c \ne 0$.
\end{enumerate}
Moreover, $f[-]: \bp_n^l(I) \to \bp_n^k$ if and only if $f \equiv c$ for
some $c > 0$. 
\end{theorem}

\begin{proof}
Clearly, $(3) \Rightarrow (1) \Rightarrow (2)$. We first show that $(2)
\Rightarrow (3)$ if $I = [0,R)$. Suppose first $f(0) = 0$. Observe that
$f[-] : \bp_n^{k+1}(I) \to \br_n^k$, so for all $a \in I \cap
[0,\infty)$, we have $f[a \Id_{k+1} \oplus\ {\bf 0}_{(n-k-1) \times
(n-k-1)}] \in \br_n^k$ if $(2)$ holds. Thus its leading principal $(k+1)
\times (k+1)$ minor vanishes, which shows that $f \equiv 0$ on $I \cap
[0,\infty)$, which contradicts (2) if $I = [0,R)$ and $f(0) = 0$.
Therefore $f(0) \ne 0$.
Now the $(1) \Rightarrow (3)$ implication of Proposition
\ref{prop:general_rank_reduction} shows that $(f - f(0))[-]:
\bp_{n-1}^l(I) \to \br_{n-1}^{k-1}$. It follows from the argument above
that $f - f(0) \equiv 0$ on $I = [0,R)$. This proves $(2) \Rightarrow
(3)$ if $I = [0,R)$. 

Now suppose $I = (-R,R)$ and $k < n-1$. Given $a \in [0,R)$, define the
following matrices:
\[
A_a := a \Id_{k+1} \oplus\ {\bf 0}_{(n-k-1) \times (n-k-1)} \in
\bp_n^l(I), \qquad \widetilde{A}_a := \begin{pmatrix} A_a & -A_a \\ -A_a
& A_a \end{pmatrix}.
\]

\noindent Note that $\widetilde{A}_a$ is the Kronecker product of $B :=
\begin{pmatrix}1 & -1 \\ -1 & 1\end{pmatrix}$ and $A_a$, so its
eigenvalues are the products of the eigenvalues of $B$ and $A_a$. It
follows that $\widetilde{A}_a \in \bp_{2n}^l(I)$. Now assume $(2)$ holds.
By Lemma \ref{lem:increase_dim}, we have $f[-]: \bp_{2n}^l(I) \to
\br_{2n}^k$. Thus, $f[\widetilde{A}_a] \in \br_{2n}^k$.
Therefore by Lemma \ref{lem:rank_minors}, the $(k+1) \times (k+1)$
principal minors of both $f[A_a]$ and $f[-A_a]$ vanish. It follows that
$f \equiv 0$ on $I = (-R,R)$, proving (3) if $I = (-R,R)$ and $f(0) = 0$.
Now suppose $f(0) \ne 0$. Then Proposition
\ref{prop:general_rank_reduction} shows that $(f - f(0))[-]:
\bp_{n-1}^l(I) \to \br_{n-1}^{k-1}$ and the result follows. 

Finally, if $f[-]: \bp_n^l(I) \to \bp_n^k$, then in particular $f[-]:
\bp_n^l(I) \to \br_n^k$ and so $f \equiv c$. It follows easily that $c >
0$. 
\end{proof}

\subsection{The regime $l \leq k < 2l$}\label{Sk<2l}

We now study the regime $l \leq k < 2l$. Theorem \ref{Tlto2l} classifies
all continuous functions $f$ which send $\bp_n^l$ into $\br_n^k$ with $l
\leq k < 2l$ and yields the same classification as in Theorem
\ref{Tltok}.

\begin{theorem}\label{Tlto2l}
Fix $0 < R \leq \infty$ and integers $2 < l < k < \min(2l,n-1)$. Suppose
$I = [0, R)$ or $(-R,R)$, and $f : I \to \R$ is continuous. The following
are equivalent:
\begin{enumerate}
\item $f[A] \in \br_n^k$ for every $A \in \br_n^l(I)$; 
\item $f[A] \in \br_n^k$ for every $A \in \bp_n^l(I)$;
\item There exists $c_1 \in \R$ such that $f(x) = f(0) + c_1 x$ for all
$x \in I$. 
\end{enumerate}
Moreover, $f[-] : \bp_n^l(I) \to \bp_n^k$ if and only if $f(0) \geq 0$
and $f(x) = f(0) + c_1 x$ for some $c_1 \geq 0$ and all $x \in I$. Also,
if $f[-] : \bp_n^l(I) \to \bp_n^k$, $I = [0,R)$, and $k \leq n-3$, the
continuity assumption is not required. 
\end{theorem}

Theorem \ref{Tlto2l} addresses the general case $2 < l < k <
\min(2l,n-1)$. Whether or not the theorem holds for the remaining
possible values for $l,k$, and $n$ is discussed in Remark
\ref{rem:Tlto2l_other_cases}.

\begin{remark}
Note that when $1 \leq k < 2l$, the combinatorial condition
\eqref{eqn:sum_binom} implies that $f$ is constant or linear. Theorems
\ref{Tk<l} and \ref{Tlto2l} thus show that the $C^k$ assumption in
Theorem \ref{Tltok} can be replaced by continuity when $1 \leq k < 2l$,
without changing the characterization. 
\end{remark}

In order to prove Theorem \ref{Tlto2l}, we need the following two
preliminary results. 

\begin{lemma}\label{lem:block_sum}
Let $A, B, C$ be three positive semidefinite matrices of dimension $n_1,
n_2$, and $n_3$ respectively. Then the $2(n_1 + n_2 + n_3) \times 2(n_1 +
n_2 + n_3)$ matrix 
\begin{equation}
M := \begin{pmatrix}
A \oplus B \oplus C & A \oplus -B \oplus C \\
A \oplus -B \oplus C & A \oplus B \oplus C
\end{pmatrix}
\end{equation}
is positive semidefinite and $\rk M = \rk A + \rk B + \rk C$.
(Recall here that $A \oplus B \oplus C = \block{A,B,C}$ denotes a block
diagonal matrix.)
\end{lemma}

\begin{proof}
To compute $\rk M$, note that the first half and the second half of the
columns of $M$ are linearly dependent. It follows easily that $\rk M =
\rk A + \rk B + \rk C$. To prove that $M$ is positive semidefinite,
suppose first that $A, B$, and $C$ are invertible. Clearly, $A \oplus B
\oplus C$ is positive definite. The Schur complement of the (2,2) block
$A \oplus B \oplus C$ in $M$ is 
\begin{align*}
S &= (A \oplus B \oplus C) - (A \oplus -B \oplus C)(A^{-1} \oplus B^{-1}
\oplus C^{-1})(A \oplus -B \oplus C) \\
&= (A \oplus B \oplus C) - (\Id_{n_1} \oplus -\Id_{n_2} \oplus \Id_{n_3})
(A \oplus -B \oplus C) \\
&= {\bf 0}_{(n_1 + n_2 + n_3)\times (n_1 + n_2 + n_3)}. 
\end{align*}
It follows that $M$ is positive semidefinite (see \cite[Appendix
A.5.5]{BoydVan_Convex}). Finally, if $A$, $B$, or $C$ is not invertible,
then the result follows by replacing $(A,B,C)$ by $(A + \epsilon
\Id_{n_1}, B + \epsilon \Id_{n_2}, C + \epsilon \Id_{n_3})$ for $\epsilon
> 0$, applying the above argument to the resulting block matrix $M$, and
letting $\epsilon \to 0^+$. 
\end{proof}

The next preliminary result demonstrates that applying $\phi_1$ entrywise
to rank $l$ matrices can double the rank. 

\begin{proposition}\label{Pphi1}
Let $0 < R \leq \infty$ and let $I = (-R,R)$. Fix integers $n > l \geq 2$
and $1 \leq k < \max(2l,n)$. Then there exists a matrix $A \in
\bp_n^l(I)$ such that $\phi_1[A] \not\in \br_n^k$.
\end{proposition}

\begin{proof}
It suffices to prove the result for $I = \R$ since $\phi_\alpha(ax) =
a^\alpha \phi_\alpha(x)$ for all $a > 0$. Suppose first that $l > k$. Let
$A \in \bp_n^l([0,R))$ have rank exactly $l$. Then $\phi_1[A] = A \not\in
\br_n^k$ and so $\phi_1$ does not map $\bp_n^l(I)$ to $\br_n^k$ if $l >
k$.  Now suppose $2 \leq l \leq k < \max(2l,n)$. Let 
\begin{equation}
A_4 := \begin{pmatrix}
8 & 4 & -2 & 6 \\
4 & 4 & 2 & 4 \\
-2 & 2 & 5 & 0 \\
6 & 4 & 0 & 5
\end{pmatrix}, \qquad 
A_6 := \begin{pmatrix}
9 & 7 & 7 & 1 & 1 & -3 \\
7 & 6 & 5 & 3 & 2 & -2 \\
7 & 5 & 6 & -1 & 0 & 2 \\
1 & 3 & -1 & 9 & 5 & 1 \\
1 & 2 & 0 & 5 & 3 & 1 \\
-3 & -2 & -2 & 1 & 1 & 3
\end{pmatrix}.
\end{equation}

\noindent It is not difficult to verify that $A_4 \in \bp_4^2(\R)$, $A_6
\in \bp_6^3(\R)$, and all the leading principal minors of $\phi_1[A_4]$
and $\phi_1[A_6]$ are nonzero. 

Suppose first $l$ is even, say $l = 2a$ for some integer $a > 0$. If $n
\geq 4a = 2l$, then the matrix
\begin{equation}\label{Epadding2}
M := B_a \oplus {\bf 0}_{(n-4a)\times(n-4a)}, \qquad B_a :=
\underbrace{A_4 \oplus \cdots \oplus A_4}_{a\ {\rm times}},
\end{equation}

\noindent satisfies $M \in \bp_n^l(\R)$, and $\phi_1[M] \in \br_n^{2l}
\setminus \br_n^{2l-1}$. This proves that $\phi_1[-]$ does not send
$\bp_n^l(I)$ to $\br_n^k$ if $n \geq 2l$ and $l$ is even. Now suppose $l
< n < 2l$. Let $M$ be the leading $n \times n$ principal submatrix of
$B_a$, i.e., the submatrix formed by its first $n$ rows and columns. Then
$M \in \bp_n^l(\R)$ since $B_a$ has rank $l$. Moreover, since every
leading principal submatrix of $\phi_1[A_4]$ is nonsingular, it follows
that $\phi_1[M]$ is also nonsingular, i.e., $\phi_1[M] \in \br_n^n
\setminus \br_n^{n-1}$. This proves the result for $n < 2l$ if $l$ is
even. 

Now suppose $l$ is odd, say $l = 2a + 1$ for some integer $a \geq 0$. For
$n \geq 4a+2 = 2l$, consider the matrix $M := B_{a-1} \oplus A_6 \oplus
{\bf 0}_{(n-4a-2)\times(n-4a-2)}$.
Then $\rk M = 2(a-1) + 3 = 2a+1 = l$. Thus $M \in \bp_n^l(\R)$. However,
$\rk \phi_1[M] = 4(a-1) + 6 = 4a+2 = 2l$. This shows that $\phi_1[-]$
does not send $\bp_n^l(I)$ to $\br_n^k$ if $n \geq 4a + 2 = 2l$ and $l$
is odd. If $l < n < 2l$, then the result follows by considering a leading
principal submatrix of $B_{a-1} \oplus A_6$ as in the case of even $l$. 
\end{proof}

With the above results in hand, we can now prove Theorem \ref{Tlto2l}. 

\begin{proof}[Proof of Theorem \ref{Tlto2l}]
Clearly, $(3) \Rightarrow (1) \Rightarrow (2)$. We first show that $(2)
\Rightarrow (3)$ if $I = [0,R)$, $f(0) = 0$, and $2 < l \leq k <
\min(2l,n-1)$. This assertion clearly holds if $f$ is constant on $I$.
Now suppose $f$ is not constant on $I$. Fix $c \in I$ such that $f(c) \ne
0$ and choose arbitrary $a,b\in \R$ such that $a^2, ab, b^2 \in I$. Let
$u := (a, b)^T$ and define
\begin{align}\label{eqn:Akl}
\begin{aligned}
C_j := &\ \underbrace{u u^T \oplus \cdots \oplus u u^T}_{j\ {\rm
times}},\\
A_{k,l}(a,b,c) := &\ C_{k-l+1} \oplus c \Id_{2l-k-1} \oplus\
\mathbf{0}_{(n-k-1) \times (n-k-1)} \in \bp_n^l(I).
\end{aligned}
\end{align}

\noindent Then $f[A_{k,l}(a,b,c)] \in \br_n^k$ by hypothesis, whence its
leading principal $(k+1) \times (k+1)$ submatrix is singular, i.e.,
\begin{equation}\label{eqn:det_Ak}
f(c)^{2l-k-1} (f(a^2) f(b^2) - f(ab)^2)^{k-l+1} = 0
\end{equation}

\noindent (using that $f(0) = 0$). Since $a,b$ are arbitrary and $f(c)
\ne 0$, it follows that $f[-] : \bp_2^1(I) \to \br_2^1$. Since $f$ is not
constant, it follows by Lemma \ref{lem:continuous_rank1} that $f(x) = c_1
f_\alpha(x) = c_1 x^\alpha$ for some $c_1 \in \R$, $\alpha > 0$, and all
$x \in I = [0,R)$.

We now show that $\alpha = 1$. First, if $\alpha \notin \{ 1, \dots, n-1
\}$, then Corollary \ref{Cnonint} implies that there exists $u \in \R^n$
such that $A_u := {\bf 1}_{n \times n} + u u^T \in \bp_n((0,\infty))$ and
$f_\alpha[A_u]$ is nonsingular (which contradicts the assumptions). We
conclude that $\alpha \in \{ 1, \dots, n-1 \}$. By Theorem \ref{Tltok}
applied on $I = [0,R)$, it follows that $\binom{\alpha+l-1}{l-1} \leq k$.
If $\alpha \geq 2$ we verify that $\binom{\alpha+l-1}{l-1} \geq
\binom{2+l-1}{l-1} = l(l+1)/2 \geq 2l$ since $l > 2$. Thus, $f \equiv c_1
x$ on $I$ and $(3)$ follows. 

Now suppose that $I = [0,R)$ and $f(0) \ne 0$. Then Proposition
\ref{prop:general_rank_reduction} shows that $(f - f(0))[-]:
\bp_{n-1}^l(I) \to \br_{n-1}^{k-1}$. If $k=l$, then it follows that $f -
f(0)$ is constant by Theorem \ref{Tk<l} and so $f \equiv f(0)$. If $k >
l$, then the above reasoning shows that $f - f(0) = c_1 x$ for some $c_1
\in \R$. Thus $f(x) = f(0) + c_1 x$ for some $c_1 \in \R$ and all $x \in
I$.  

Next, we prove that $(2) \Rightarrow (3)$ when $I = (-R,R)$. The result
clearly holds if $f$ is constant. Thus, assume $f$ is nonconstant.
Suppose first $f(0) = 0$. We use a technique similar to the one used in
the proof of Theorem \ref{Tk<l} for $I = (-R,R)$. Let $c \in I$ be such
that either $f(c) \ne 0$ or $f(-c) \ne 0$ and let $a,b \in \R$ such that
$a^2, b^2, ab \in I$. Consider the matrix  
\begin{equation}
\widetilde{A} := \begin{pmatrix}A_{k,l}(a,b,|c|) & A_{k,l}(a,b,-|c|) \\
A_{k,l}(a,b,-|c|) & A_{k,l}(a,b,|c|)\end{pmatrix}, 
\end{equation}

\noindent with $A_{k,l}(x,y,z)$ as in \eqref{eqn:Akl}. By Lemma
\ref{lem:block_sum}, we have $\widetilde{A} \in \bp_{2n}^l(I)$. Also, by
Lemma \ref{lem:increase_dim}, we have $f[-]: \bp_{2n}^l(I) \to
\br_{2n}^k$. Thus, $f[\widetilde{A}] \in \br_{2n}^k$. Therefore, by Lemma
\ref{lem:rank_minors}, the $(k+1) \times (k+1)$ principal minors of
$A_{k,l}(a,b,|c|)$ and $A_{k,l}(a,b,-|c|)$ both vanish. Computing the
determinant of these two matrices as in \eqref{eqn:det_Ak}, we conclude
that $f(a^2)(b^2)-f(ab)^2 = 0$ for all $a,b \in \R$ such that $a^2, b^2,
ab \in I$. It follows by Proposition \ref{prop:charac_rank1} that $f[-] :
\bp_2^1(I) \to \br_2^1$. Since $f$ is not constant, by Lemma
\ref{lem:continuous_rank1}, we have that either $f(x) = c_1
\phi_\alpha(x)$ or $f(x) = c_1 \psi_\alpha(x)$ for some $c_1 \in \R$,
$\alpha > 0$, and all $x \in I = (-R,R)$. Now since $f[-]: \bp_n^l(I) \to
\br_n^k$, then in particular $f[-]: \bp_n^l([0,R)) \to \br_n^k$.
Therefore, by the previous part $f \equiv c_1 x$ on $[0,R)$. It follows
that $\alpha = 1$, i.e., $f \equiv c_1 \phi_1$ or $f \equiv c_1 \psi_1$
on $I = (-R,R)$. By Proposition \ref{Pphi1}, the function $\phi_1$ does
not map $\bp_n^l(I)$ into $\br_n^k$. Thus, $f(x)  = c_1 \psi_1(x) = c_1
x$ for all $x \in I$. This proves the result for $I = (-R,R)$ if $f(0) =
0$. If $f(0) \ne 0$, then applying Proposition
\ref{prop:general_rank_reduction} shows that $(f - f(0))[-]:
\bp_{n-1}^l(I) \to \br_{n-1}^{k-1}$ and the result easily follows. 

Now suppose $f[-] : \bp_n^k(I) \to \bp_n^k$. Then in particular $f[-] :
\bp_n^k(I) \to \br_n^k$ and so $f(x) = f(0) + c_1 x$ for all $x \in I$.
It follows easily that $f(0), c_1 \geq 0$. The converse is obvious.
Finally, if $k \leq n-3$ and $f[-] : \bp_n^k([0,R)) \to \bp_n^k$, then by
Theorem \ref{Thorn}, we do not need any continuity assumption on $f$. 
\end{proof}

\begin{remark}\label{rem:Tlto2l_alternate}
When $I = (-R,R)$, the proof of Theorem \ref{Tlto2l} depends on  Lemmas
\ref{lem:increase_dim} and \ref{lem:block_sum}. We now provide a direct
argument that avoids using these lemmas. Let $a,b \in \R$ such that $a^2,
b^2, ab \in I$ and let $u := (a,b)^T$. Let $A := C_{(k+1)/2} \oplus {\bf
0}_{(n-k-1) \times (n-k-1)}$ if $k$ is odd, and $A := C_{(k+2)/2} \oplus
{\bf 0}_{(n-k-2) \times (n-k-2)}$ if $k$ is even, where $C_j$ was defined
in \eqref{eqn:Akl}. Note that $n \geq k+2$ since $k < n-1$ by hypothesis
and so $A$ is well-defined. Also, since $k < 2l$, then $l \geq (k+1)/2$
if $k $ is odd, and $l \geq (k+2)/2$ if $k$ is even. Therefore, $A \in
\bp_n^l(I)$, which implies that $f[A] \in \br_n^k$ if $k$ is odd, and
$f[A] \in \br_n^{k+1}$ if $k$ is even. It follows that $f[-]: \bp_2^1(I)
\to \br_2^1$. The proof can now be concluded by using the rest of the
argument in the proof of Theorem \ref{Tlto2l}. Note that when $k$ is odd
and $n = k-1$, the proof given above is also valid.
\end{remark}

\begin{remark}\label{rem:Tlto2l_other_cases}
Various cases were left unresolved in Theorem \ref{Tlto2l} in order to
simplify the statement of the theorem. We now address each one of them
separately.\medskip

\noindent \textit{Case 1: $l=k$.}
When $l=k$ and $f(0) = 0$, the arguments used in the proof of Theorem
\ref{Tlto2l} show that $f(x) = c_1 x$ for some $c_1 \in \R$. The converse
also clearly holds. If $f(0) \ne 0$, Proposition
\ref{prop:general_rank_reduction} shows that $(f-f(0))[-]: \bp_{n-1}^l
\to \br_{n-1}^{k-1} = \br_{n-1}^{l-1}$, and Theorem \ref{Tk<l} implies
that $f \equiv f(0)$.\medskip

\noindent \textit{Case 2: $k=2l$ and $f(0) \ne 0$.}
In this case, Proposition \ref{prop:general_rank_reduction} shows that
$(f - f(0))[-]: \bp_{n-1}^l \to \br_{n-1}^{k-1} = \br_{n-1}^{2l-1}$. We
conclude by Theorem \ref{Tlto2l} that $f(x) = f(0) + c_1 x$ for some $c_1
\in \R$ and all $x \in I$.\medskip

\noindent \textit{Case 3: $(l,k) = (2,3)$.}
If $l=2$, $k=3$, and $n \geq 4$, the result becomes slightly different.
First, if $f(0) \ne 0$, then $f - f(0) \equiv c_1 x$ by Proposition
\ref{prop:general_rank_reduction} and the $l=k$ case. Now suppose $f(0) =
0$. The proof of Theorem \ref{Tlto2l} shows that either $f \equiv c_1
\phi_\alpha$ or $f \equiv c_1 \psi_\alpha$ on $I$ for some $\alpha \in
\{1,\dots,n-1\}$. If $\alpha > 2$, then $\binom{\alpha+l-1}{l-1} =
\binom{\alpha + 1}{1} = \alpha + 1 \geq 4$, and Theorem \ref{Tltok}
applied on $[0,R)$ implies that $c_1 \phi_\alpha$ or $c_1 \psi_\alpha$
cannot send $\bp_n^2(I)$ to $\br_n^3$ if $c_1 \neq 0$. Thus $\alpha = 1$
or $2$. By Theorem \ref{Tltok}, the functions $f \equiv c_1 x, c_1 x^2$
do map $\bp_n^2(I)$ to $\br_n^3$. We now claim that $\phi_1$ and $\psi_2$
don't. This is clear for $\phi_1$ by Proposition \ref{Pphi1}. To prove
that $\psi_2[-]$ does not send $\bp_n^2$ to $\br_n^3$, let $B_4 := (\cos
\frac{(i-j)\pi}{4})_{i,j=1}^4$ be the matrix constructed in \cite[Section
1]{Bhatia-Elsner}. Then one easily verifies that $\psi_2[B_4]$ is
nonsingular. Therefore, if $l=2$, $k=3$, and $f(0) = 0$, then $f \equiv
c_1 x$ or $f \equiv c_1 x^2$ on $I$.\medskip

\noindent \textit{Case 4: $k = n-1$.}
In Theorem \ref{Tlto2l}, we assume $k < \min(2l,n-1)$. It is natural to
ask if the assumption can be relaxed to $k < \min(2l,n)$. Note that when
$I = [0,R)$, the proof of Theorem \ref{Tlto2l} goes through for $k <
\min(2l,n)$. The result also holds when $I = (-R,R)$ and $k$ is odd; see
Remark \ref{rem:Tlto2l_alternate}.
However, the result fails in general when $I = (-R,R)$ and $k$ is even.
For instance, the $(2) \Rightarrow (3)$ implication in Theorem
\ref{Tlto2l} does not hold if $k=l=2$, $n=3$, and $I = (-R,R)$. In fact,
we claim that every function $f$ such that $f \equiv 0$ on $[-R/2,R)$
automatically satisfies $(2)$ when $n=3$. To show the claim, first
suppose that at least one of $r,s,t$ lies in $[-R/2,R)$. Now given any
matrix
\[
A = \begin{pmatrix} a & r & s\\ r & b & t\\s & t & c\end{pmatrix} \in
\bp_3^2(I),
\]

\noindent it is clear that $\det f[A] = 2 f(r) f(s) f(t) = 0$, so $f[A]
\in \br_3^2$ and $(b)$ holds.

Suppose instead that $r,s,t \in (-R, -R/2)$. We show that $A$ (of the
above form) cannot lie in $\bp_3^2(I)$ - in fact, not even in $\bp_3(I)$.
(This shows more generally that if $B \in \bp_n(I)$, then for every
principal $3 \times 3$ submatrix $A$ of $B$, at least one off-diagonal
entry lies in $[-R/2,R)$.) First compute:
\begin{align*}
4\det A = &\ 4abc - 4at^2 - 4bs^2 - 4cr^2 + 8rst \leq 4abc - (a+b+c)R^2 -
R^3\\
< &\ 4 (a+b+c)^3/27 - (a+b+c)R^2 - R^3,
\end{align*}

\noindent by the arithmetic mean-geometric mean inequality.
Let $u := (a+b+c)/3$ and define $g(x) := 4 x^3 - 3xR^2 - R^3$. Note that
if $A \in \bp_3^2$ then $u \in [0,R)$.
The above computations thus show that if $r,s,t \in (-R,-R/2)$, then $4
\det A < g(u)$, with $u \in [0,R)$. Note that $g'(x) = 12(x^2-(R/2)^2)$,
which is nonpositive on $[0,R/2]$ and positive on $(R/2,R)$. Thus $g(x)$
is decreasing on $[0,R/2]$ and increasing on $[R/2,R]$; moreover, $g(0) <
0 = g(R)$. Hence we get
\[
4 \det A < g(u) \leq 0,
\]

\noindent which shows that if $r,s,t \in (-R,-R/2)$ then $A \notin
\bp_3$.
\end{remark}

In Case 4 of Remark \ref{rem:Tlto2l_other_cases}, we demonstrated that if
$k=l=2$, $n=3$, and $I = (-R,R)$, any function $f$ such that $f \equiv 0$
on $[-R/2,R)$ maps $\bp_3^2(I)$ into $\br_3^2$. We now prove that the
conclusion of Theorem \ref{Tlto2l} holds if $f \not\equiv 0$ on
$[-R/2,R)$ and $k < \min(2l,n)$. Note that all the cases where $k < n-1$
have already been considered in Theorem \ref{Tlto2l} under more general
hypotheses. 

\begin{theorem}\label{Tlto2l_ext}
Fix $0 < R \leq \infty$ and integers $2 < l < k < \min(2l,n)$. Suppose $I
= (-R,R)$, $f : I \to \R$ is continuous, and $f \not\equiv 0$ on
$[-R/2,R)$. Then the following are equivalent:
\begin{enumerate}
\item $f[A] \in \br_n^k$ for every $A \in \br_n^l(I)$; 
\item $f[A] \in \br_n^k$ for every $A \in \bp_n^l(I)$;
\item There exists $c_1 \in \R$ such that $f(x) = f(0) + c_1 x$ for all
$x \in I$. 
\end{enumerate}
Moreover, $f[-] : \bp_n^l(I) \to \bp_n^k$ if and only if $f(0) \geq 0$
and $f(x) = f(0) + c_1 x$ for some $c_1 \geq 0$ and all $x \in I$. Also,
if $f[-] : \bp_n^l(I) \to \bp_n^k$, $I = [0,R)$, and $k \leq n-3$, the
continuity assumption is not required. 
\end{theorem}

\begin{proof}
Clearly $(3) \Rightarrow (1) \Rightarrow (2)$. The implication $(2)
\Rightarrow (3)$ was already proved in greater generality when $k < n-1$
or $k = n-1$ and $k$ is odd (see Theorem \ref{Tlto2l} and Remark
\ref{rem:Tlto2l_other_cases}). Thus, it suffices to prove $(2)
\Rightarrow (3)$ under the assumption that $k = n-1$ and $k$ is even. 

Let $2 < l < n-1 < 2l$, let $k := n-1$ be even, and suppose $(2)$ holds.
Also, suppose first that $f(0) = 0$. We consider two cases:\medskip

\noindent \textit{Case 1: $f[-]: \bp_2^1(I) \to \br_2^1$.}
In this case, following the argument in the proof of Theorem
\ref{Tlto2l}, we conclude that $f(x) = c_1 x$ for all $x \in I$ and some
$c_1 \in \R$.\medskip

\noindent \textit{Case 2: $f[-]: \bp_2^1(I) \not\to \br_2^1$.}
In this case, there exists $\widehat{A_2} \in \bp_2^1(I)$ such that
$f[\widehat{A_2}]$ has rank $2$. Note that $l \geq (n+1)/2$ and so $f[-]:
\bp_n^{\frac{n+1}{2}}(I) \to \br_n^{n-1}$. Also note that $n \geq 5$ and
$n$ is odd, say $n = 3 + 2a$ for some $a \geq 1$. Now given $A_3 \in
\bp_3^2(I)$, consider the matrix 
\begin{equation*}
A = A(\widehat{A_2},A_3) := A_3 \oplus \underbrace{\widehat{A_2} \oplus
\cdots \oplus \widehat{A_2}}_{a\ {\rm times}} \in \bp_n^{2+a}(I) =
\bp_n^{\frac{n+1}{2}}(I). 
\end{equation*}

\noindent Since $f[A(\widehat{A_2},A_3)] \in \br_n^{n-1}$ for all $A_3
\in \bp_3^2(I)$, it follows that $f[-]: \bp_3^2(I) \to \br_3^2$. Since $f
\not\equiv 0$ on $[-R/2,R)$, we now claim that there exists $c \in (0,R)$
such that $f(c) \neq 0$. The claim is clear if $f \not\equiv 0$ on
$(0,R)$; otherwise let $x_0 \in (-R/2,0)$ such that $f(x_0) \neq 0$.
Define the block diagonal matrix
\begin{align}\label{Elmatrix}
B(x_0) :=  x_0 {\bf 1}_{3 \times
3} - 3 x_0 \Id_3 \in \bp_3^2(I),
\end{align}

\noindent Since $f[-]: \bp_3^2(I) \to \br_3^2$, the matrix $f[B(x_0)]$
has determinant zero, i.e.,
\[
0 = \det f[B(x_0)] = \left(f(-2 x_0) + 2 f(x_0)\right)\left(f(-2x_0) -
f(x_0)\right)^2.
\]

\noindent Hence either $f(-2x_0) = -2 f(x_0) \ne 0$ or $f(-2x_0) = f(x_0)
\ne 0$ This proves the claim with $c := -2x_0$. Now considering the
matrix $A_{2,2}(a,b,c)$ as in \eqref{eqn:Akl}, we conclude that $f[-]:
\bp_2^1(I) \to \br_2^1$ and so $f(x) = c_1 x$ for some $c_1 \in \R$ as in
Case 1. This proves the result if $f(0) = 0$. Note that the proof goes
through for $f(0) = 0$ even if $l = k$. 

If $f(0) \ne 0$, then applying Proposition
\ref{prop:general_rank_reduction} shows that $(f - f(0))[-]:
\bp_{n-1}^l(I) \to \br_{n-1}^{k-1}$ and the result easily follows using
the above analysis.\medskip

Now suppose $f[-] : \bp_n^l(I) \to \bp_n^k$. Then in particular $f[-] :
\bp_n^l(I) \to \br_n^k$ and so $f(x) = f(0) + c_1 x$ for all $x \in I$.
It follows easily that $f(0), c_1 \geq 0$. The converse is obvious.
Finally, if $k \leq n-3$ and $f[-] : \bp_n^k([0,R)) \to \bp_n^k$, then by
Theorem \ref{Thorn}, we do not need any continuity assumption on $f$. 
\end{proof}

\section{Preserving positivity and absolutely monotonic functions}\label{Sam}

In the final section of this paper, we return to the original problem
studied by Schoenberg, Rudin, Horn, Vasudeva, Hiai and others, i.e., the
characterization of entrywise functions mapping $\bp_n(I)$ to $\bp_n$ for
all $n$.
We demonstrate a stronger result, namely, that preserving positivity on
$\bp_n^2$ (in fact on all special rank $2$ matrices) for all $n \geq 1$
is equivalent to preserving positivity on $\bp_n$ for all $n \geq 1$. In
particular, we provide a new proof of a generalization of the result by
Vasudeva \cite{vasudeva79}. 

First recall some classical results about absolutely monotonic functions.
Define the $m$-th forward difference of a function $f$, with step $h>0$
at the point $x$, to be
\begin{equation}
\Delta^m_h[f](x) := \sum_{i=0}^m (-1)^i \binom{m}{i} f(x+(m-i)h). 
\end{equation}

We now state two important results about absolutely monotonic functions
which will be needed to prove Theorem \ref{thm:rank2_AM}.

\begin{theorem}[see {\cite[Chapter IV, Theorem
7]{widder}}]\label{thm:abs_monotonic_equiv}
Let $0 < R \leq \infty$ and let $f: [0,R) \rightarrow \R$. Then the
following are equivalent:
\begin{enumerate}
\item The function $f$ is absolutely monotonic on $[0,R)$.
\item The function $f$ can be extended analytically to the disc $D(0,R)
\subset \mathbb{C}$, and $f(z) = \sum_{i=0}^\infty a_i z^i $ for some
$a_i \geq 0$.
\item For every $m \geq 1$, $\Delta^m_h[f](x) \geq 0$ for all $x$ and $h$
such that 
\begin{equation*}
0 \leq x < x+h < \dots < x+mh < R. 
\end{equation*}
\end{enumerate}
\end{theorem}

Theorem \ref{thm:abs_monotonic_equiv} can be used to show the following
useful result.

\begin{lemma}\label{lem:limit_abs_monotonic}
Let $0 < R \leq \infty$ and let $f_n: [0,R) \rightarrow \R$, $n \geq 1$,
be a sequence of absolutely monotonic functions and assume $f_n(x)
\rightarrow f(x)$ for every $x \in [0,R)$. Then $f$ is absolutely
monotonic on $[0,R)$.
\end{lemma}

\begin{proof}
By Theorem \ref{thm:abs_monotonic_equiv}, the forward differences
$\Delta^k_h[f_n](x)$ of $f_n$ are nonnegative for all integers $n \geq 0$
and for all $x$ and $h$ such that $0 \leq x < x+h < \dots < x+nh < R$.
Since $f$ is the pointwise limit of the sequence $f_n$, the same is true
for $\Delta^k_h[f](x)$. As a consequence, by Theorem
\ref{thm:abs_monotonic_equiv}, the function $f$ is absolutely monotonic.
\end{proof}

Recall that Vasudeva \cite{vasudeva79} proved that functions mapping all
positive semidefinite matrices with positive entries into themselves are
absolutely monotonic (see Theorem \ref{vasudeva_rank2}). Theorem
\ref{thm:rank2_AM} strengthens Vasudeva's result by working only with
special rank $2$ matrices. As we noticed before, Theorem
\ref{thm:rank2_AM} follows immediately from Theorem \ref{Thorn} (see
Remark \ref{Thorn_implies_vasudeva}). Using the techniques developed
above, we now provide a more transparent and elementary proof of Theorem
\ref{thm:rank2_AM}. 

\begin{proof}[Proof of Theorem \ref{thm:rank2_AM}]
That $(3) \Rightarrow (2)$ follows from the Schur product theorem, and
clearly $(2) \Rightarrow (1)$. We now show that $(1) \Rightarrow (3)$.
Assume that $f \in C^\infty(I)$ and let $0 < a < R$.
Define $f_a: [0, R-a) \rightarrow \mathbb{R}$ by $f_a(x) := f(a+x)$. Then
\begin{equation}\label{eqn:f_a}
f_a[A] \in \bp_n \textrm{ for every } A \in \bp_n^1([0, R-a)). 
\end{equation}

Denote by $\{m_1, m_2,\dots, m_{k(a)}\} \subset \{0,1,\dots,n-1\}$ the
(possibly empty) set of indices between $0$ and $n-1$ such that
$f_a^{(i)}(0) \not= 0$ if and only if $i = m_j$ for some $j$. By Theorem
\ref{thm:sum_powers}(3), we have $f_a^{(m_j)}(0) > 0$ for all $1 \leq j
\leq k(a)$. Consequently, $f^{(i)}(a) = f_a^{(i)}(0)\geq 0$ for all $0
\leq i \leq n-1$, for all $n \in \N$, and all $a \in I$. Since $f$ is
smooth, it follows from \cite[Chapter IV]{widder} that $f$ has a power
series representation with nonnegative coefficients, which proves the
result when $f \in C^\infty(I)$.

Now assume $f$ is not necessarily smooth and let $0 < b < R$. First note
by Step 3 of the proof of Theorem \ref{Thorn} that $f$ is continuous on
$I$. Now given any probability distribution $\theta \in
C^\infty(b/R,\infty)$ with compact support in $(b/R, \infty)$, let 
\begin{equation}
f_\theta(x) := \int_{b/R}^\infty f(xy^{-1}) \theta(y) \frac{dy}{y},
\qquad 0 < x < b.
\end{equation}

\noindent Then $f_\theta \in C^\infty(0,b)$. Suppose $A \in
\bp_n^1(0,b)$. Then, for every $\beta \in \mathbb{R}^n$, 
\begin{align*}
\beta^T (f_\theta)[A]\beta = &\  \sum_{i,j=1}^{n}  \beta_i
\beta_j\int_{b/R}^\infty f(a_{ij}y^{-1}) \theta(y)\frac{dy}{y} = \ \
\int_{b/R}^\infty \sum_{i,j=1}^{n}  \beta_i \beta_j f(a_{ij}y^{-1})
\theta(y)\frac{dy}{y} \\
= &\ \ \int_{b/R}^\infty \beta^T f[y^{-1}A] \beta \theta(y)\frac{dy}{y} .
\end{align*}

\noindent Note that the integrand is nonnegative for every $y > 0$. It
thus follows that $\beta^T (f_\theta)[A] \beta \geq 0$. Since $\beta$ is
arbitrary, $f_\theta[A] \in \bp_n$. 

Now consider a sequence $\theta_m \in C^\infty(\R)$ of probability
distributions with compact support in $(b/R,\infty)$ such that $\theta_m$
converges weakly to $\delta_1$, the Dirac measure at $1$. Note that such
a sequence can be constructed since $b/R < 1$. By the first part of the
proof, $f_{\theta_m}$ is absolutely monotonic on $(0,b)$ for every $m
\geq 1$.
Therefore by Theorem \ref{thm:abs_monotonic_equiv}, the forward
differences $\Delta^k_h[f_{\theta_m}](x)$ of $f_{\theta_m}$ are
nonnegative for $l \geq 0$ and all $x$ and $h$ such that $0 \leq x < x+h
< \dots < x+lh < R$. 
Since $f$ is continuous, $f_{\theta_m}(x) \rightarrow f(x)$ for every $x
\in (0,b)$. Therefore $\Delta^k_h[f](x) \geq 0$ for all such $x$ and $h$
as well. As a consequence, by Theorem \ref{thm:abs_monotonic_equiv}, the
function $f$ is absolutely monotonic on $(0,b)$. Since this is true for
every $0 < b < R$, it follows that $f$ is absolutely monotonic on $I$.
\end{proof}

It is natural to ask if results similar to Theorem \ref{thm:rank2_AM}
hold when $I = [0,R)$. In other words, can one characterize functions
that preserve positivity for all positive semidefinite matrices, or for
positive semidefinite matrices of rank at most 2? Before answering this
question, we first point out a subtle difference between functions that
are absolutely monotonic on $(0,R)$ and $[0,R)$.

\begin{remark}\label{rem:abs_mon_at_0}
Recall that $f$ is absolutely monotonic on $[0,R)$ if and only if
its derivatives are all nonnegative on $(0,R)$ and $f$ is continuous at
$0$. If instead $f: [0,R) \to \R$ satisfies $f[A] \in \bp_n$ for all $A
\in \bp_n([0,R))$, then $f$ is absolutely monotonic, nonnegative, and
nondecreasing on $(0,R)$.
Therefore $f$ has (at most) a removable discontinuity at $0$. Redefining
$f$ at $0$ to be $\lim_{x \rightarrow 0^+} f(x)$, we get that $f$ is
absolutely monotonic on $[0,R)$.
\end{remark}

We now prove two characterization results analogous to Theorem
\ref{thm:rank2_AM}, but for $I = [0,R)$.

\begin{theorem}\label{thm:special2ToPn}
Suppose $0 < R \leq \infty$, $I = [0,R)$, and $f: I \rightarrow
\mathbb{R}$. Then the following are equivalent:
\begin{enumerate}[(a)]
\item For all $n \geq 1$, $f[a {\bf 1}_{n \times n} + u u^T] \in \bp_n$
for every $a \in I$ and $u \in [0, \sqrt{R-a})^n$.

\item The function $f$ is absolutely monotonic on $(0,R)$ and $0 \leq
f(0) \leq f^+(0) := \lim_{x \to 0^+} f(x)$.
\end{enumerate}

Similarly, the following are equivalent.
\begin{enumerate}[(1)]
\item $f$ is right-continuous at $0$, and for all $n \geq 1$, $f[a {\bf
1}_{n \times n} + u u^T] \in \bp_n$ for every $a \in I$ and $u \in
[0, \sqrt{R-a})^n$.
\item For all $n \geq 1$, $f[A] \in \bp_n$ for every $A \in \bp_n(I)$.
\item The function $f$ is absolutely monotonic on $I$.
\end{enumerate}
\end{theorem}

In particular, the result is more involved for $I = [0,R)$ than for $I =
(0,R)$, since functions preserving the set of matrices of the form $a
{\bf 1}_{n \times n} + u u^T$ need not preserve all matrices in $\bp_n$
for all $n$.

\begin{proof}
If (a) holds, then $f$ is absolutely monotonic on $(0,R)$ by Theorem
\ref{thm:rank2_AM}, hence nonnegative and nondecreasing. Thus the
right-hand limit of $f$ at zero, $f^+(0)$, exists and is nonnegative. Now
define the matrix $A_t := \displaystyle \begin{pmatrix} t & 0\\0 &
0\end{pmatrix} \in \bp_2^1$. Then $\lim_{t \to 0^+} f[A_t] \in \bp_2$,
which implies: $0 \leq f(0) \leq f^+(0)$, proving (b).

Conversely, suppose (b) holds. If $a > 0$, then (a) follows from the $(3)
\Rightarrow (1)$ part of Theorem \ref{thm:rank2_AM}.
Now, assume $a=0$. Suppose $0 \leq f(0) < f^+(0)$ and let $A = u u^T$ for
some $u \in [0,\sqrt{R})^n$. By permuting rows and columns of $A$, we can
assume it is of the form
\begin{equation}
A = \begin{pmatrix}
A' & \mathbf{0}_{n_1 \times n_2} \\
\mathbf{0}_{n_2 \times n_1} & \mathbf{0}_{n_2 \times n_2}
\end{pmatrix}, 
\end{equation}

\noindent where $0 \leq n_1, n_2 \leq n$, and $A'$ has no zero entry. By
equation \eqref{Eboyd}, the matrix $f[A]$ is positive semidefinite if and
only if $f(0) \geq 0$ (which holds by hypothesis) and $(f - f(0))[A'] \in
\bp_n$.
Now note by Remark \ref{rem:abs_mon_at_0} that the function
$\widetilde{f} : [0,R) \rightarrow \R$ obtained by redefining $f$ at $0$
to be $f^+(0)$ is absolutely monotonic on $[0,R)$, and so (1) holds for
$\widetilde{f}$ by the Schur product theorem. Hence the function
$\widetilde{f} - f^+(0)$ is absolutely monotonic on $(0,R)$. Now since $0
\leq f(0) < f^+(0)$ by assumption,
\[
(f - f(0))[A'] = (\widetilde{f} - f^+(0))[A'] + (f^+(0) - f(0)) {\bf
1}_{n \times n} \in \bp_n.
\]

\noindent This implies $f[A] \in \bp_n$ and concludes the proof of the
first equivalence.

We now show the second set of equivalences. That (3) $\Rightarrow$
(1),(2) follows from the right continuity of $f$ at $0$ and the Schur
product theorem as in the $I = (0,R)$ case. If (1) holds, then $f$ is
continuous at $0$, as well as absolutely monotonic on $(0,R)$ by Theorem
\ref{thm:rank2_AM}, which shows (3). We finally show that $(2)
\Rightarrow (1)$. As in the proof of the first equivalence, $0 \leq f(0)
\leq f^+(0)$. Now proceed as in the proof of Theorem \ref{thm:unif} to
conclude that $f^+(0) = f(0)$. 
\end{proof}

\begin{remark}
Vasudeva's proof of Theorem \ref{vasudeva_rank2} can be adapted to
functions with possibly finite domains $f : (0,R) \to \infty$ for $0 < R
\leq \infty$. 
However, Vasudeva's methods do not extend to studying the problem of
preserving positivity with rank constraints.
In contrast, we solve the harder rank-constrained problem for
\textit{fixed} dimension by using a novel approach as described in
Section \ref{S3step}. As a consequence, we prove Theorem
\ref{vasudeva_rank2} as a special case of Theorem \ref{thm:rank2_AM}, and
moreover, by a more intuitive proof than that in \cite{vasudeva79}.
\end{remark}

\section*{Acknowledgment}

The authors would like to thank the referee(s) for going through the
paper in great detail and for providing numerous useful comments and
suggestions. A part of this work was undertaken when the third author was
visiting the University of Sydney.

\subsection*{List of edits to the published version:}
\begin{enumerate}
\item Definition~\ref{D11}: Added the first sentence.

\item Statement of Theorem~\ref{Tltok} --
There are three changes:
(i)~$k$ can be $0$, including in the final sentence;
(ii)~the polynomial $f$ in part~(2) can have a constant term; and
(iii)~the $i_t$ are distinct (and non-negative).

\item The proof of Proposition~\ref{Tvandermonde} was earlier termed
``Proof of Theorem~\ref{Tvandermonde}.'' This is now corrected.

\item Statement of Proposition~\ref{P49}: The integer $m$ is now
specified to be positive.

\item Proof of Proposition \ref{Pgvm}, line 3: The definition of the set
$C$ uses $\mathbb{Q}^n$ -- this should be $\mathbb{Q}^l$. Similarly, the
Cartesian product is over $l$ factors, not $n$.

\item Statements of Lemma~\ref{lem:increase_dim} and
Corollary~\ref{cor:sp_rank2_neg}: The value of $k$ is now allowed to be
zero.

\item Proof of Theorem~\ref{Tltok}:
\begin{enumerate}
\item The first two sentences are new.

\item The vectors ${\bf m}_j$ now comprise the set of all vectors in
$(\Z_{\geq 0} \cap [0, k-1])^l$.

\item The third and fourth sentences after Equation~\eqref{E52} have been
somewhat modified.
\end{enumerate}

\item (For completeness:) While the authors' affiliations have since
changed, they are retained below as they are in the published paper.
\end{enumerate}

\end{document}